\documentclass[11pt,a4paper]{amsart}  

\usepackage[english]{babel}

\usepackage{amsmath}
\usepackage{amsthm}
\usepackage{amsfonts}
\usepackage{dsfont}
\usepackage{enumitem}
\usepackage{color}
\usepackage{graphicx}
\usepackage{fullpage}

\newcommand{\reals}{\mathbb{R}}
\newcommand{\rationals}{\mathbb{Q}}
\newcommand{\complex}{\mathbb{C}}

\newcommand{\integers}{\mathbb{Z}}

\newcommand{\bracketa}[1]{\big[#1\big]}
\newcommand{\bracketb}[1]{\Big[#1\Big]}
\newcommand{\bracketc}[1]{\bigg[#1\bigg]}

\newcommand{\paraa}[1]{\big(#1\big)}
\newcommand{\parab}[1]{\Big(#1\Big)}
\newcommand{\parac}[1]{\bigg(#1\bigg)}

\newcommand{\spacearound}[1]{\quad#1\quad}

\renewcommand{\implies}{\spacearound{\Rightarrow}}
\newcommand{\qtext}[1]{\quad\text{#1}\quad}
\newcommand{\qqtext}[1]{\qquad\text{#1}\qquad}
\newcommand{\qand}{\qtext{and}}
\newcommand{\qqand}{\qqtext{and}}

\newcommand{\mathand}{\qand}

\newtheorem{theorem}{Theorem}[section]

\newtheorem{proposition}[theorem]{Proposition}

\theoremstyle{definition}
\newtheorem{definition}[theorem]{Definition}
\theoremstyle{remark}
\newtheorem{remark}[theorem]{Remark}
\numberwithin{equation}{section}

\newcommand{\xv}{\vec{x}}

\newcommand{\C}{\mathcal{C}}
\newcommand{\Ch}{\C_{\hbar}}
\newcommand{\Chp}{\C_{\hbar'}}
\newcommand{\Chinf}{\Ch^{\infty}}
\newcommand{\Chhinf}{\widehat{\mathcal{C}}_{\hbar}^{\infty}}
\newcommand{\Chpinf}{\Chp^{\infty}}

\newcommand{\Der}{\operatorname{Der}}

\newcommand{\g}{\mathfrak{g}}

\newcommand{\eps}{\varepsilon}

\newcommand{\E}{\mathcal{E}}
\newcommand{\Ehhp}{\E_{\hbar,\hbar'}}
\newcommand{\lhp}[1]{_\bullet\langle#1\rangle}
\newcommand{\rhp}[1]{\langle#1\rangle_\bullet}
\newcommand{\xh}{\hat{x}}
\newcommand{\prodh}{\bullet_\hbar}
\newcommand{\fh}{\hat{f}}
\newcommand{\gh}{\hat{g}}

\newcommand{\yv}{\vec{y}}
\newcommand{\zv}{\vec{z}}
\newcommand{\sigmah}{\sigma_{\hbar}}
\newcommand{\A}{\mathcal{A}}
\newcommand{\TChhinf}{T\Chhinf}

\newcommand{\Mphi}{M_{\varphi}}
\newcommand{\lambdat}{\tilde{\lambda}}

\newcommand{\epst}{\tilde{\eps}}

\newcommand{\gt}{\tilde{g}}
\newcommand{\ft}{\tilde{f}}
\newcommand{\supp}{\operatorname{supp}}
\newcommand{\nablat}{\widetilde{\nabla}}

\title[]{Projections, modules and connections \\ for the noncommutative cylinder}
\date{August 2020}
\author{Joakim Arnlind and Giovanni Landi}

\address[Joakim Arnlind]{Department of Mathematics\\
Link\"oping University\\
581 83 Link\"oping\\
Sweden}
\email{joakim.arnlind@liu.se}

\address[Giovanni Landi]{
Matematica, Universit\`a di Trieste \\ 
Via A.Valerio, 12/1, 34127  Trieste, Italy \\
Institute for Geometry and Physics (IGAP) Trieste, Italy \\
and INFN, Trieste, Italy
}
\email{landi@units.it}

\begin{document}

\begin{abstract}
  We initiate a study of projections and modules over a noncommutative
  cylinder, a simple example of a noncompact noncommutative
  manifold. Since its algebraic structure turns out to have many
  similarities with the noncommutative torus, one can develop several
  concepts in a close analogy with the latter. In particular, we
  exhibit a countable number of nontrivial projections in the algebra
  of the noncommutative cylinder itself, and show that they provide
  concrete representatives for each class in the corresponding $K_0$
  group. We also construct a class of bimodules endowed with
  connections of constant curvature.  Furthermore, with the
  noncommutative cylinder considered from the perspective of
  pseudo-Riemannian calculi, we derive an explicit expression for the
  Levi-Civita connection and compute the Gaussian curvature.
\end{abstract}

\maketitle

\tableofcontents
\parskip = .75ex

\section{Introduction}

\noindent
In the rapidly developing and conceptually growing
field of noncommutative geometry it has been of paramount importance
to have at least one tractable example exhibiting many of the nontrivial
subtleties of the theory. In this respect, 
the noncommutative torus is perhaps the most studied object in
noncommutative geometry, and it has served as a inspirational source
(as well as testing ground) for many results and concepts in more
general situations.  However, to explore the notion of noncompact
manifolds, the torus is not equally well suited. 

In this paper, we set
out to study the noncommutative cylinder as a simple manageable example of a
noncompact noncommutative manifold which still exhibits nontrivial
features. Inspired by the algebraic similarities with the torus, we
follow the same lines of thought in order to see to what extent known
concepts apply in this noncompact situation as well.

Starting from a known description in terms of Fourier transforms, we choose a particular presentation of
the noncommutative cylinder and introduce a (commuting) set of
hermitian derivations as well as a trace. After providing basic
results about these structures, we proceed to construct a class of
projections   in the algebra itself, and show that they are classified by
the integers. Moreover, by showing that the corresponding projective
modules respect the group structure of the integers, we conclude that
these projections   provide concrete representatives for each class in
the $K_0$ group of the noncommutative cylinder (which is known to be
$\integers$). A corresponding ``Chern number'' can be computed for
each projective module by evaluating the projections against a cyclic
2-cocycle.

Next, in analogy with the torus modules defined by Connes and Rieffel
\cite{c:cstaralgebre,r:irrational.rotation}, we find a class of
bimodules for the noncommutative cylinder, on which connections of
constant curvature are defined. Interestingly, these modules turn out
to be isomorphic to copies of the algebra itself. Although the details
of the bimodule structure depend on a choice of parameters, it is the case 
that the curvature only depends on the deformation parameters
$\hbar$ and $\hbar'$ defining the left and right algebras,
respectively.

Finally, we recall the framework of pseudo-Riemannian calculi, and
show that for a given choice of metric, there exists a calculus over
the noncommutative cylinder with a unique torsion-free and metric
connection, for which one may explicitly compute the Gaussian
curvature. Moreover, we illustrate a Gauss-Bonnet type theorem where
the total curvature (that is, the integral of the Gaussian curvature with
respect to the Riemannian volume form) is shown to be independent of a
class of metric perturbations.

\section{The algebra of the noncommutative cylinder}

\noindent
Let us start by recalling the definition of the algebra of the
noncommutative cylinder.  Let $\mathcal{S}(\reals\times S^1)$ denote the space of
Schwartz functions on $\reals\times S^1$. Every
$f\in\mathcal{S}(\reals\times S^1)$ may be written as
\begin{align}\label{eq:f.fourier.series}
  f(u,t) = \sum_{n\in\integers}f_n(u)e^{2\pi int},
\end{align}
with $f_n\in\mathcal{S}(\reals)$ and we introduce the Fourier
transform of the coefficients $f_n$ as
\begin{align*}
  \fh_n(x) = \int_\reals f_n(u)e^{-2\pi i ux}du.
\end{align*}
Thus any function as in \eqref{eq:f.fourier.series}, is written as 
\begin{align*}
  f(u,t) = \sum_{n\in\integers} \int_\reals \fh_n(x) e^{2\pi i (nt+ux)}dx.
\end{align*}
Following the general strategy of \cite{r:deformation.quantization.Rd}, we define a twisted convolution
product on $\mathcal{S}(\reals\times S^1)$ via
\begin{align}\label{eq:prodh.def}
  \widehat{(f\prodh g)}_n(x) =
  \sum_{k\in\integers}\int_{\reals}\fh_k(y)\gh_{n-k}(x-y)\sigmah(\yv,\xv-\yv)dy
\end{align}
where $\xv=(x,n)$, $\yv=(y,k)$ and $\sigmah$ is a cocycle
fulfilling the condition
\begin{align*}
  \sigmah(\xv,\yv)\sigmah(\xv+\yv,\zv)
  =\sigmah(\xv,\yv+\zv)\sigmah(\yv,\zv),
\end{align*}
ensuring associativity of the product. For our purposes we choose
a particular cocycle given by
\begin{align}\label{eq:sigma.cocycle}
  \sigmah\paraa{(x,n),(y,k)} = e^{2\pi i\hbar yn}.
\end{align}
Note that this cocycle is cohomologous to its antisymmetrization
$$
\sigmah(\xv,\yv)=e^{\pi i\hbar(yn-x k)},
$$ 
giving the corresponding
twisted convolution as defined in \cite{vs:nc.cylinder}; the two corresponding algebras are thus isomorphic.

\begin{definition}
  Let $\Chinf = (\mathcal{S}(\reals\times S^1), \prodh)$ be the algebra defined by the vector space
  $\mathcal{S}(\reals\times S^1)$ together with the product  $\prodh$ in
  \eqref{eq:prodh.def} for the cocycle
  \begin{equation*}
    \sigmah\paraa{(x,n),(y,k)} = e^{2\pi i\hbar yn}.
  \end{equation*}
\end{definition}

\noindent
As the product in $\Chinf$ is defined on the level of Fourier
transforms, let us derive a more explicit expression in the following
form.

\begin{proposition}\label{prop:prodh}
  Let $f,g\in \Chinf$ be such that
  \begin{align*}
    f(u,t) = \sum_{n\in\integers}f_n(u)e^{2\pi i nt}\qqand
    g(u,t) = \sum_{n\in\integers}g_n(u)e^{2\pi i nt}.
  \end{align*}
  Then
  \begin{align*}
    (f\prodh g)(u,t) = \sum_{n\in\integers}\bracketc{
    \sum_{k\in\integers}f_k(u)g_{n-k}(u+k\hbar)
    }e^{2\pi i nt}.
  \end{align*}
\end{proposition}

\begin{proof}
  The proof consists of a straight-forward computation:
  \begin{align*}
  (f\prodh g)_n(u)
    &= \int_{\reals}\widehat{f\prodh g}_n(x)e^{2\pi ix u}dx\\
    &= \sum_{k\in\integers}\int_\reals\int_\reals\fh_k(y)\gh_{n-k}(x-y)
      e^{2\pi i\hbar(x-y)k}e^{2\pi ix u}dydx\\
    &=\sum_{k\in\integers}\int_\reals\fh_k(y)
      \bracketc{\int_{\reals}\gh_{n-k}(x-y)e^{2\pi ix(u+k\hbar)}dx}
      e^{-2\pi iyk\hbar}dy\\
    &= \sum_{k\in\integers}\int_\reals\fh_k(y)e^{2\pi iy(u+k\hbar)}
    \bracketc{\int_{\reals}\gh_{n-k}(x)e^{2\pi ix(u+k\hbar)}dx}e^{-2\pi iyk\hbar}dy\\
    &=\sum_{k\in\integers}\int_\reals\fh_k(y)g_{n-k}(u+k\hbar)e^{2\pi iyu}dy
      =\sum_{k\in\integers}f_k(u)g_{n-k}(u+k\hbar).\qedhere
  \end{align*}
\end{proof}

\noindent
From Proposition~\ref{prop:prodh} one infers the simple commutation
rule
\begin{align}\label{eq:tu.simple.com}
  f(u)e^{2\pi int}\prodh g(u) = f(u)g(u+n\hbar)\prodh e^{2\pi int}
\end{align}
which we shall often use in the following. 
To slightly simplify the notation, let us introduce $W=e^{2\pi it}$
such that every $f\in\mathcal{S}(\reals\times S^1)$ may be written as
\begin{align*}
  f(u,t) = \sum_{n\in\integers}f_n(u)W^n.
\end{align*}
In particular, \eqref{eq:tu.simple.com} now reads
\begin{align}\label{eq:tu.simple.com-2}
W^n \prodh f(u) = f(u+n\hbar)\prodh W^n .
\end{align}

\begin{remark}
As a side remark, we note that the relation
\eqref{eq:tu.simple.com} can formally be derived from the canonical
commutation relation $[u,t]=i\hbar/2\pi$ via
\begin{align*}
  e^{2\pi it}u
  &= \sum_{n\geq 0}\frac{(2\pi it)^nu}{n!}
  =\sum_{n\geq 0}\frac{(2\pi i)^n(ut^n-in\hbar t^{n-1}/2\pi)}{n!}\\
  &=ue^{2\pi it}+\hbar\sum_{k\geq 1}\frac{(2\pi it)^{n-1}}{(n-1)!}
    =(u+\hbar)e^{2\pi it},
\end{align*}
which is in close analogy with the noncommutative catenoid defined in
\cite{ah:catenoid}.
\end{remark}

\noindent
One may readily introduce a $\ast$-algebra structure on $\Chinf$.
\begin{proposition}\label{prop:star.structure}
  For $f = \sum_{n\in\integers}f_n(u)W^n\in\Chinf$, set
  \begin{align*}
    f^\ast = \sum_{n\in\integers}\overline{f_{n}(u-n\hbar)}W^{-n}
    = \sum_{n\in\integers}\overline{f_{-n}(u+n\hbar)}W^n.
  \end{align*}
Then it follows that $(f^\ast)^\ast = f$ and
  $(f\prodh g)^\ast=g^\ast\prodh f^\ast$.
\end{proposition}

\begin{proof}
Just compute
  \begin{align*}
    (f^\ast)^\ast
    &= \sum_{n\in\integers}\overline{f^\ast_{-n}(u+n\hbar)}W^n\\
    &=\sum_{n\in\integers}\overline{\overline{f_{-(-n)}(u+n\hbar-n\hbar)}}W^n
      =\sum_{n\in\integers}f_n(u)W^n = f.
  \end{align*}
  Next, consider
  \begin{align*}
    f = \sum_{n\in\integers}f_n(u)W^n\qand
    g = \sum_{n\in\integers}g_n(u)W^n
  \end{align*}
  and compute
  \begin{align*}
    (g\prodh f)^\ast
    &=\sum_{n\in\integers}\overline{(g\prodh f)_{-n}(u+n\hbar)}W^n\\
    &=\sum_{n,k\in\integers}\overline{g_k(u+n\hbar)}\,\overline{f_{-n-k}(u+(n+k)\hbar)}W^n\\
    &=\sum_{n,l\in\integers}\overline{f_{-l}(u+l\hbar)}\,\overline{g_{-(n-l)}(u+n\hbar)}W^n\\
    &=\sum_{n,l\in\integers}\overline{f_{-l}(u+l\hbar)}\,\overline{g_{-(n-l)}(u+l\hbar+(n-l)\hbar)}W^n\\
    &=\sum_{n,l\in\integers}f^\ast_{l}(u)g^\ast_{n-l}(u+l\hbar)W^n=f^\ast\prodh g^\ast,
  \end{align*}
  which proves the second statement.
\end{proof}

\noindent
Thus, with respect to the involution defined in
Proposition~\ref{prop:star.structure} the algebra $\Chinf$ is a
$\ast$-algebra. By representing $\Chinf$ as multiplication operators
on $L^2(\reals\times S^1)$, i.e. $T_f\psi(u,t)=f(u,t)\psi(u,t)$
  for $f\in\Chinf$, one may complete $\Chinf$ in the operator norm to
a $C^\ast$-algebra which we shall denote by $\Ch$
(cf. \cite{vs:nc.cylinder}).

Next, let us introduce a set of commuting derivations.

\begin{proposition}
  For $f\in\Chinf$ with
  \begin{align*}
    f(u,t) = \sum_{n\in\integers}f_n(u)W^n
  \end{align*}
  define
  \begin{align*}
    \partial_1f = \sum_{n\in\integers}f'_n(u)W^n\quad\textrm{and}\quad
      \partial_2f = 2\pi i\sum_{n\in\integers}nf_n(u)W^n.    
  \end{align*}
  Then $\partial_1$ and $\partial_2$ are hermitian derivations of $\Chinf$ such
  that $[\partial_1,\partial_2]=0$.
\end{proposition}

\begin{proof}
  It is clear that $\partial_1$ and $\partial_2$ are linear maps; let us show that
  they satisfy Leibniz rule. One obtains
  \begin{align*}
    \partial_1(f\prodh g)
    &= \sum_{n,k\in\integers}
    \parab{f'_k(u)g_{n-k}(u+\hbar k)+f_k(u)g'_{n-k}(u+k\hbar)}W^n\\
    &=\sum_{n,k\in\integers}f'_k(u)g_{n-k}(u+\hbar k)W^n
      +\sum_{n,k\in\integers}f_k(u)g'_{n-k}(u+k\hbar)W^n\\
    &=(\partial_1f)\prodh g + f\prodh(\partial_1g),
  \end{align*}
  and
  \begin{align*}
    \partial_2(f\prodh g)
	 &=2\pi i\sum_{n,k\in\integers}nf_k(u)g_{n-k}(u+k\hbar)W^n\\
    &=2\pi i\sum_{n,k\in\integers}kf_k(u)g_{n-k}(u+k\hbar)W^n\\
	  &\quad +2\pi i\sum_{n,k\in\integers}f_k(u)(n-k)g_{n-k}(u+k\hbar)W^n\\
         &=(\partial_2f)\prodh g + f\prodh(\partial_2g),          
  \end{align*}
  showing that $\partial_1,\partial_2$ are indeed derivations of
  $\Chinf$. Furthermore, it is easy to see that
  \begin{align*}
    [\partial_1,\partial_2](f)
    &= 2\pi i\partial_1\sum_{n\in\integers}nf_n(u)W^n
      -\partial_2\sum_{n\in\integers}f_n'(u)W^n=0.
  \end{align*}
  Finally, let us show that $\partial_1$ and $\partial_2$ are hermitian
  derivations. One computes
  \begin{align*}
    \paraa{\partial_1(f)}^\ast
    &=\parab{\sum_{n\in\integers}f'_n(u)W^n}^\ast
      =\sum_{n\in\integers}\overline{f'_{-n}(u+n\hbar)}W^n\\
    &= \partial_1\sum_{n\in\integers}\overline{f_{-n}(u+n\hbar)}W^n
      =\partial_1(f^\ast)
  \end{align*}
  as well as
  \begin{align*}
    \paraa{\partial_2(f)}^\ast
    &= \parab{2\pi i\sum_{n\in\integers}nf_n(u)W^n}^\ast
      =-2\pi i\sum_{n\in\integers}\overline{(-n)f_{-n}(u+n\hbar)}W^n\\
    &= 2\pi i\sum_{n\in\integers}n\overline{f_{-n}(u+n\hbar)}W^n
      =\partial_2\sum_{n\in\integers}\overline{f_{-n}(u+n\hbar)}W^n
      =\partial_2(f^\ast)
  \end{align*}
  which proves that $\partial_1,\partial_2$ are hermitian.
\end{proof}

\begin{remark}
Clearly the function $u$ is not in the algebra $\Chinf$. In spite of this a direct computation shows that one can formally obtain 
a commutation expression for the derivation $\partial_2$, that is
\begin{align}\label{der-com}
 \partial_2 f = \frac{2\pi i}{\hbar} \, (f u - u f) 
\end{align}
for any $f\in\Chinf$. 
\end{remark}

\noindent
On the algebra $\Chinf$ we have a trace as well.
\begin{definition}\label{def:tau.def}
  For $f\in\Chinf$ with
  \begin{align*}
    f(u,t) = \sum_{n\in\integers}f_n(u)W^n
  \end{align*}
  we set
  \begin{align}\label{eq:tau.def}
    \tau(f) = \int_{-\infty}^\infty f_0(u)du.    
  \end{align}
\end{definition}

\noindent It is clear from the definition that $\tau$ is a linear map.

\begin{proposition}\label{prop:tau.properties}
  The map $\tau$ is a positive invariant trace; that is, it has the properties
  \begin{enumerate}
  \item $\tau(f^\ast)=\overline{\tau(f)}$,
  \item $\tau(f^\ast\prodh f)\geq 0$,
  \item $\tau(f\prodh g)=\tau(g\prodh f)$,
  \item $\tau(\partial_1f)=\tau(\partial_2f)=0$,
  \end{enumerate}
  for all $f,g\in\Chinf$.
\end{proposition}

\begin{proof}
  It is immediate to see that
  $\tau(f^\ast)=\overline{\tau(f)}$. A direct computation yields
  \begin{align*}
    \tau(f\prodh g)
    &=\int_{\reals}\sum_{k\in\integers}f_{k}(u)g_{-k}(u+k\hbar)du
      =\int_{\reals}\sum_{k\in\integers}g_{k}(u-k\hbar)f_{-k}(u)du\\
    &=\int_{\reals}\sum_{k\in\integers}
      g_k(v)f_{-k}(v+k\hbar)dv=\tau(g\prodh f).
  \end{align*}
  Furthermore, one finds that
  \begin{align*}
    \tau(\partial_1f) = \int_{\reals}f'_0(u)du = \bracketa{f_0(u)}_{-\infty}^\infty=0
  \end{align*}
  as well as
  \begin{align*}
    \tau(\partial_2f) = \tau\parac{\sum_{n\in\integers}nf_n(u)W^n}
    =\int_\reals 0\cdot f_0(u)du = 0.
  \end{align*}
  Finally, we check that
  \begin{align*}
    \tau(f^\ast\prodh f)
    &= \sum_{k\in\integers}f^\ast_{k}f_{-k}(u+k\hbar)
      = \sum_{k\in\integers}|f_{-k}(u+k\hbar)|^2\geq 0,
  \end{align*}
  which completes the proof of the statements.
\end{proof}

\noindent
From now on we shall drop the cumbersome notation $f\prodh g$ and
simply write $fg$ when no confusion can arise.

For the noncommutative torus, there exists a convenient cyclic
2-cocycle which can be evaluated on 2-forms. For the noncommutative
cylinder, one can make use of a similar construction. The cyclic
2-cocycle below will be used in the next section in order to compute
``Chern numbers'' of a class of projective modules.

\begin{proposition}\label{prop:cyclic.cocycle}
  For $f_0,f_1,f_2\in\Chinf$ we set
  \begin{align*}
    \Psi(f_0,f_1,f_2) = \frac{1}{2\pi i}\tau\paraa{f_0(\partial_1f_1)(\partial_2f_2)-f_0(\partial_2f_1)(\partial_1f_2)}.
  \end{align*}
  Then $\Psi$ is a cyclic,  
  \begin{equation*}
    \Psi(f_2,f_0,f_1)=\Psi(f_0,f_1,f_2),
  \end{equation*}
 Hochschild 2-cocycle,
  \begin{align*}
    \Psi(f_0f_1,f_2,f_3)-\Psi(f_0,f_1f_2,f_3)+\Psi(f_0,f_1,f_2f_3)-\Psi(f_3f_0,f_1,f_2) = 0 ,
  \end{align*}
  for all $f_0,f_1,f_2,f_3\in\Chinf$.
\end{proposition}

\begin{proof}
  Let us first show that $\Psi$ is cyclic.
  By using
  $\tau(fg)=\tau(gf)$ one finds that
  \begin{align*}
	  &\quad\ 2\pi i\Psi(f_2,f_0,f_1)
    = \tau\bracketb{f_2(\partial_1f_0)(\partial_2f_1)-f_2(\partial_2f_0)(\partial_1f_1)}\\
           &= \tau\bracketb{(\partial_1f_2f_0)(\partial_2f_1)-(\partial_2f_2f_0)(\partial_1f_1)-(\partial_1f_2)f_0(\partial_2f_1)+(\partial_2f_2)f_0(\partial_1f_1)}\\
           &=\tau\bracketb{(\partial_1f_2f_0)(\partial_2f_1)-(\partial_2f_2f_0)(\partial_1f_1)}+2\pi i\Psi(f_0,f_1,f_2),
  \end{align*}
  and since $\tau(\partial_1f)=\tau(\partial_2f)=0$ (by Proposition~\ref{prop:tau.properties}) it follows that
  \begin{align*}
    2\pi i\Psi(f_2,f_0,f_1)
    &=2\pi i\Psi(f_0,f_1,f_2)
      -\tau\bracketb{f_2f_0(\partial_1\partial_2f_1-\partial_2\partial_1f_1)}\\
    &=2\pi i\Psi(f_0,f_1,f_2)
  \end{align*}
  since $[\partial_1,\partial_2]=0$. To show that $\Psi$ is a cocycle, i.e. 
  \begin{align*}
    \Psi(f_0f_1,f_2,f_3)-\Psi(f_0,f_1f_2,f_3)+\Psi(f_0,f_1,f_2f_3)-\Psi(f_3f_0,f_1,f_2) = 0,
  \end{align*}
  is a straight-forward computation where one expands all derivatives
  of products of functions, and uses the fact that
  $\tau(fg)=\tau(gf)$.
\end{proof}

\section{Projections in the algebra}

\noindent
For the noncommutative torus, it is well known that
its algebra, in contrast to the commutative case,
contains nontrivial projections   which one may explicitly describe
\cite{r:irrational.rotation}. In this section, we will show that a
similar construction can be carried out for the noncommutative
cylinder. Namely, we shall construct projections   $p\in\Chinf$ of the
following form:
\begin{align*}
  p = g(u+\hbar)W + f(u) + g(u)W^{-1}.
\end{align*}

\begin{proposition}\label{prop:fg.conditions}
  Let $f,g\in \mathcal{S}(\reals)$ be real-valued functions, and set
\begin{align*}
  p = g(u+\hbar)W + f(u) + g(u)W^{-1}.
\end{align*}
Then $p^\ast=p$. Moreover $p^2=p$ if the functions $f$ and $g$ satisfy
\begin{align}
  &g(u)g(u+\hbar) = 0\label{eq:psq.gghbar}\\
  &g(u)\paraa{1-f(u)-f(u-\hbar)} = 0\label{eq:psq.gff}\\
  &g(u)^2+g(u+\hbar)^2=f(u)-f(u)^2\label{eq:psq.ggff}
\end{align}
for all $u\in\reals$.
\end{proposition}

\begin{proof}
  Since $f$ and $g$ are real-valued, using \eqref{eq:tu.simple.com-2} one immediately obtains
  \begin{align*}
    p^\ast = W^{-1}g(u+\hbar)+f(u)+Wg(u)
    =g(u)W^{-1}+f(u)+g(u+\hbar)W = p.
  \end{align*}
  Then, a straight-forward computation of $p^2$ gives
  \begin{align*}
	  p^2 &= g(u)g(u-\hbar)W^{-2}+\paraa{f(u)g(u)+g(u)f(u-\hbar)}W^{-1}\\
        &\quad +\paraa{f(u)g(u+\hbar)+g(u+\hbar)f(u+\hbar)}W+g(u+\hbar)g(u+2\hbar)W^2\\
        &\quad +g(u+\hbar)^2+f(u)^2+g(u)^2
  \end{align*}
  which indeed equals $p$ by using \eqref{eq:psq.gghbar}--\eqref{eq:psq.ggff}.
\end{proof}

\noindent
Let us now construct a particular class of projections satisfying the
requirements of Proposition~\ref{prop:fg.conditions}. Let $f_0$ be a
function increasing smoothly from $0$ to $1$ on the interval
$[0,\hbar]$, and define $f,g:\reals\to\reals$ as
\begin{align}
  &f(u) =
  \begin{cases}
    0 &\text{if }u\leq 0\text{ or }u\geq 2\hbar\\
    f_0(u) &\text{if }0\leq u\leq \hbar\\
    1-f_0(u-\hbar) &\text{if }\hbar\leq u\leq 2\hbar
  \end{cases}\label{eq:f.projector}\\
  &g(u) =
  \begin{cases}
    0 &\text{if }u\leq\hbar\text{ or }u\geq 2\hbar\\
    \sqrt{f(u)-f(u)^2} &\text{if }\hbar\leq u\leq 2\hbar.
  \end{cases}\label{eq:g.projector}
\end{align}
Next, for $n\geq 1$ we set
\begin{align*}
  &f_n(u) = \sum_{k=1}^nW^{-2k}f(u)W^{2k}\\
  &g_n(u) = \sum_{k=1}^nW^{-2k}g(u)W^{2k},
\end{align*}
resulting in $n$ shifted copies of the original functions, as depicted
in Figure~\ref{fig:fandg}. Note that $f_n$ and $g_n$ have compact
support being defined on $[0,2\hbar n]$, where they are
$2\hbar$-periodic by construction.
 \begin{figure}[h]
  \centering
  \includegraphics[width=\textwidth]{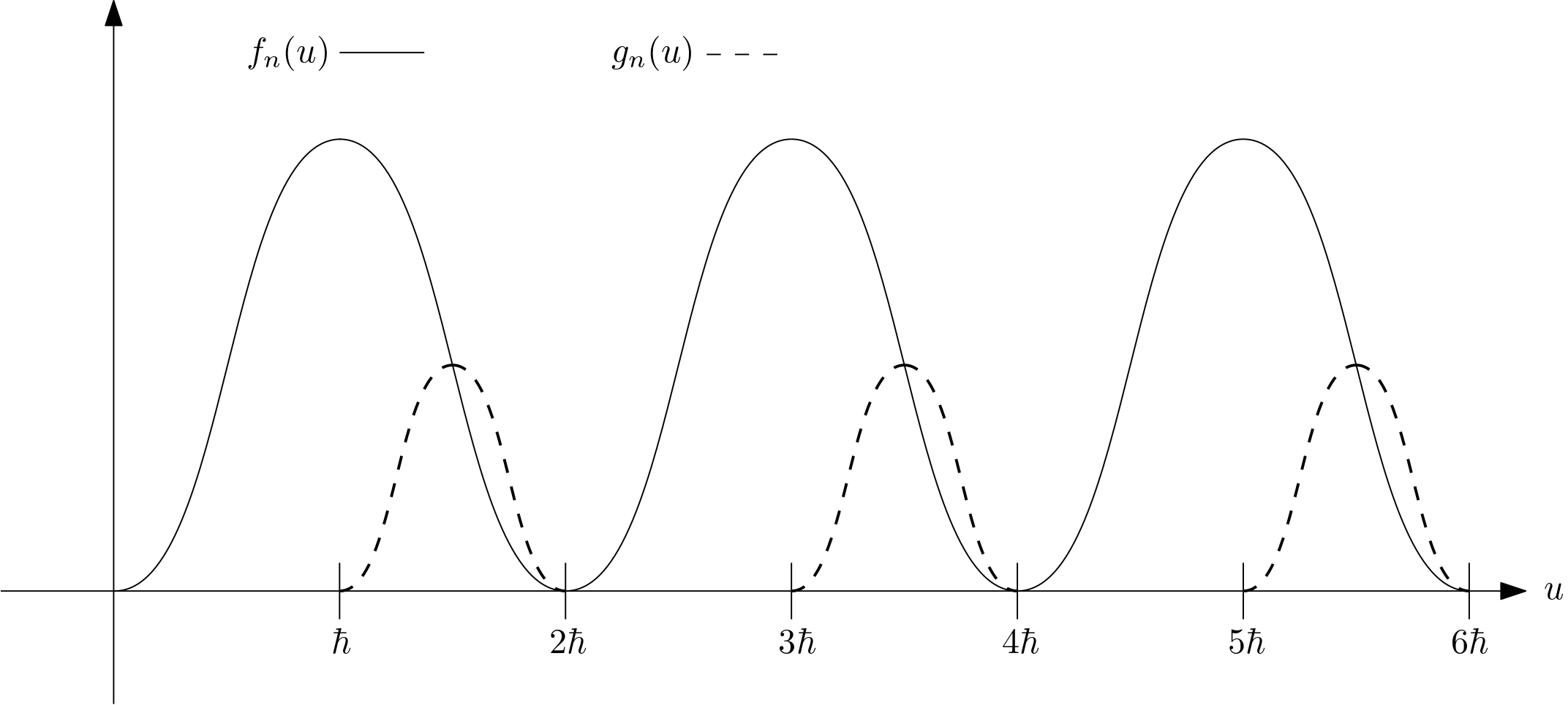}
  \caption{The functions $f_n$ and $g_n$ as constructed from \eqref{eq:f.projector} and \eqref{eq:g.projector}.}
  \label{fig:fandg}
\end{figure}

It is straightforward to check that $f_n$ and $g_n$ satisfy
\eqref{eq:psq.gghbar}, \eqref{eq:psq.gff} and \eqref{eq:psq.ggff}. For
instance, for $u\in[0,\hbar]$ it is immediate that
\eqref{eq:psq.gghbar} and \eqref{eq:psq.gff} holds since
$g(u)=0$. Moreover,
\begin{align*}
  g(u+\hbar)^2 &= f(u+\hbar)-f(u+\hbar)^2=1-f_0(u)-(1-f_0(u))^2\\
               &=f_0(u)-f_0(u)^2=f(u)-f(u)^2,
\end{align*}
showing that \eqref{eq:psq.ggff} is satisfied as well. Thus, one may
conclude from Proposition~\ref{prop:fg.conditions} that
\begin{align}
  p_n = g_n(u+\hbar)W + f_n(u) + g_n(u)W^{-1}\label{eq:def.pn}
\end{align}
is indeed a projection in $\Chinf$. Next, let us compute the trace of these projections.

\begin{proposition}\label{prop:trace.pn}
  Let $p_n$ be defined as above. Then $\tau(p_n)=n\hbar$.
\end{proposition}

\begin{proof}
  Since $f_n$ is supported on $[0,2\hbar n]$, where it is
  $2\hbar$-periodic, it follows that
  \begin{align*}
    \tau(p_n) = \tau\parab{g_n(u+\hbar)W+f_n(u)+g_n(u)W^{-1}}= n\int_{0}^{2\hbar}f_n(u)du,
  \end{align*}
  and from the definition of $f_n$ one obtains
  \begin{align*}
    \tau(p_n) &= n\int_0^{\hbar}f_0(u)du + n\int_{\hbar}^{2\hbar}\paraa{1-f_0(u-\hbar)}du\\
	      &= n\int_{\hbar}^{2\hbar}du = n\hbar.\\[-4em]
  \end{align*}
\end{proof}

\noindent 
The curvature 2-form related to the projection $p_n$ is given by
$F_n=p_ndp_ndp_n$, which may be evaluated against the cyclic 2-cocycle
defined in Proposition~\ref{prop:cyclic.cocycle}.

\begin{proposition}\label{prop:cyclic.cocycle.pn}
For any projection $p_n$ as in \eqref{eq:def.pn}, one has $$\Psi(p_n,p_n,p_n) = n.$$
\end{proposition}

\begin{proof}
  As
  \begin{align*}
    &\Psi(p_n,p_n,p_n)
    =\frac{1}{2\pi i}\tau\paraa{p_n(\partial_1p_n)(\partial_2p_n)-p_n(\partial_2p_n)(\partial_1p_n)}    
  \end{align*}
  we compute
  \begin{align*}
    p_n &= g_n(u+\hbar)W + f_n(u) + g_n(u)W^{-1} \\
    \partial_1p_n &= g_n'(u+\hbar)W + f_n'(u) + g_n'(u)W^{-1}\\              
    \partial_2p_n &=2\pi i g_n(u+\hbar)W - 2\pi i g_n(u)W^{-1}.
  \end{align*}
  Writing
  \begin{align*}
    p_n(\partial_1p_n)(\partial_2p_n)-p_n(\partial_2p_n)(\partial_1p_n)
    =\sum_{n\in\integers}A_n(u)W^n
  \end{align*}
  one finds that
  \begin{align*}
    -\frac{1}{2\pi i} A_0 &= f_n(u)\parab{g_n(u+\hbar)Wg_n'(u)W^{-1}-g_n(u)W^{-1}g_n'(u+\hbar)W}\\
     &\quad -g_n(u+\hbar)Wg_n(u)W^{-1}f_n'(u)+g_n(u)W^{-1}g_n(u+\hbar)Wf_n'(u)\\
     &\quad -f_n(u)\parab{-g_n'(u+\hbar)Wg_n(u)W^{-1}+g_n'(u)W^{-1}g_n(u+\hbar)W}\\
     &\quad +g_n(u+\hbar)Wf_n'(u)g_n(u)W^{-1}-g_n(u)W^{-1}f_n'(u)g_n(u+\hbar)W\\
	  & = f_n(u)\parab{2g_n(u+\hbar)g_n'(u+\hbar)-2g_n(u)g_n'(u)} -g_n(u+\hbar)^2f_n'(u)\\
	  &\quad +g_n(u)^2f_n'(u)
       +g_n'(u+\hbar)^2f_n'(u+\hbar)-g_n(u)^2f_n'(u-\hbar)\\
	  &= \paraa{f_n(u)g_n(u+\hbar)^2-f_n(u)g_n(u)^2}'_u
       -2g_n(u+\hbar)^2f_n'(u)\\
	  &\quad +2g_n(u)^2f_n'(u) +g_n(u+\hbar)^2f_n'(u+\hbar)-g_n(u)^2f_n'(u-\hbar),
  \end{align*}
  giving
  \begin{align*}
    \Psi (p_n,p_n,p_n)=\frac{1}{2\pi i}&\tau(A_0)
	  =\tau\Big[ 2g_n(u+\hbar)^2f_n'(u)+g_n(u)^2f_n'(u-\hbar)\\
	  &\qquad -2g_n(u)^2f_n'(u)-
	   g_n(u+\hbar)^2f_n'(u+\hbar)\Big]\\
         &=3\!\int_{-\infty}^\infty g_n(u)^2f_n'(u-\hbar)du-\!3\!\int_{-\infty}^\infty g_n(u)^2f_n'(u)du.
  \end{align*}
  Since $g_n(u)=0$ for all $u\in[2\pi k,2\pi k+\hbar]$ and
  $f_n(u)=1-f_n(u-\hbar)$ for all $u\in[2k\hbar+\hbar,2k\hbar+2\hbar]$ for
  $k=0,\ldots,n-1$,
  \begin{align*}
    \int_{-\infty}^\infty g_n(u)^2f_n'(u)du=-\int_{-\infty}^\infty g_n(u)f'_n(u-\hbar)du,
  \end{align*}
  and it follows that
  \begin{align*}
    \Psi(p_n,p_n,p_n) &= 6\int_{-\infty}^\infty g_n(u)^2f'_n(u-\hbar)du
    =6n\int_{\hbar}^{2\hbar}g(u)^2f'(u-\hbar)du.
  \end{align*}
  Noting that for $u\in[\hbar,2\hbar]$
  \begin{align*}
	  g(u)^2 = f(u)-f(u)^2&=1-f(u-\hbar)-(1-f(u-\hbar))^2\\
	  &=f(u-\hbar)-f(u-\hbar)^2
  \end{align*}
  one computes
  \begin{align*}
    \Psi(p_n,p_n,p_n)
    &= 6n\int_{\hbar}^{2\hbar}\paraa{f(u-\hbar)-f(u-\hbar)^2}f'(u-\hbar)du\\
      &=6n\int_0^1(s-s^2)ds=6n\bracketc{\frac{1}{2}-\frac{1}{3}}=n,
  \end{align*}
  which proves the statement.
\end{proof}

\noindent
Considering the construction of the projection $p_n$, and the results
in Proposition~\ref{prop:trace.pn} and
Proposition~\ref{prop:cyclic.cocycle.pn}, it is natural to ask how the
direct sum of the projective modules defined by $p_n$ and $p_m$ is
related to the module defined by $p_{m+n}$. The next result shows that
they are indeed isomorphic.

\begin{proposition}\label{prop:pn.pm.sum}
  Let $n,m$ be integers with $n,m\geq 1$. Then
  \begin{align*}
    p_n\Chinf\oplus p_m\Chinf\simeq p_{n+m}\Chinf
  \end{align*}
  as (right) $\Chinf$-modules.
\end{proposition}

\begin{proof}
  Let $p_n$ and $p_m$ be given as
  \begin{align*}
    &p_n=g_n(u+\hbar)W+f_n(u)+g_n(u)W^{-1}\\
    &p_m=g_m(u+\hbar)W+f_m(u)+g_m(u)W^{-1}
  \end{align*}
  and introduce
  \begin{align*}
	  \tilde{p}_m = W^{-2n}p_mW^{2n} &= g_m(u+\hbar-2n\hbar)W+f_m(u-2n\hbar)+g_m(u-2n\hbar)W^{-1}.
  \end{align*}
  Since $\tilde{p}_m$ is unitarily equivalent to $p_m$, the modules
  $p_m\Chinf$ and $\tilde{p}_m\Chinf$ are isomorphic and, furthermore, it is
  clear that $p_n+\tilde{p}_{m} = p_{n+m}$. Next, let us show that $p_n$
  and $\tilde{p}_m$ are orthogonal; i.e. that $p_n\tilde{p}_m=0$.
  Introduce
  \begin{align*}
    &\gt_m(u) = g_m(u-2n\hbar)\\
    &\ft_m(u) = f_m(u-2n\hbar)
  \end{align*}
  and note that
  \begin{align*}
    g_n(u)\gt_m(u)=f_n(u)\ft_m(u)=g_n(u)\ft_m(u)=f_n(u)\gt_m(u)=0
  \end{align*}
  since $\supp(f_n,g_n)\subseteq(0,2n\hbar)$ and
  $\supp(\ft_m,\gt_m)\subseteq(2n\hbar,2(n+m)\hbar)$ are
  disjoint. Using these facts, one finds that
  \begin{align*}
    p_n\tilde{p}_m
    &= g_n(u+\hbar)\gt_m(u+2\hbar)W^2 + f_n(u)\gt_m(u+\hbar)W\\
      &\quad +g_n(u)\ft_m(u-\hbar)W^{-1}+g_n(u)\gt_m(u-\hbar)W^{-2}.
  \end{align*}
  First of all, it is clear that $g_n(u)\ft_m(u-\hbar)=0$ and
  $g_n(u)\gt_m(u-\hbar)=0$ since $u-\hbar<2n\hbar$ whenever
  $u\in\supp g_n\subseteq(0,2n\hbar)$. Furthermore, it also follows
  that $g_n(u+\hbar)\gt_m(u+2\hbar)=0$ and $f_n(u)\gt_m(u+\hbar)=0$
  since $\gt_m(u)=0$ for $u\in[2n\hbar,(2n+1)\hbar]$. Thus, we
  conclude that $p_n\tilde{p}_m=0$.  For orthogonal projections, 
  \begin{align*}
    p_n\Chinf\oplus\tilde{p}_m\Chinf\simeq (p_n+\tilde{p}_m)\Chinf,
  \end{align*}
  and in combination with the previous arguments one obtains
  \begin{align*}
    p_n\Chinf\oplus p_m\Chinf\simeq
    p_n\Chinf\oplus \tilde{p}_m\Chinf\simeq
    (p_n+\tilde{p}_m)\Chinf \simeq p_{n+m}\Chinf,
  \end{align*}
  which proves the desired result.
\end{proof}

\noindent
Let us discuss these results from the perspective of $K$-theory. In
\cite{vs:nc.cylinder}, $K_0$ of the noncommutative cylinder was shown
to be isomorphic to $\integers$. Since the algebra is nonunital, one
then expects that there exists a countable class of nontrivial
projections  . In this section we have constructed projections   $p_n$ (for
each $n\geq 1$), and Proposition~\ref{prop:trace.pn} shows that if
$m\neq n$ then $p_n$ and $p_m$ are not equivalent. Moreover, from
Proposition~\ref{prop:pn.pm.sum} it follows that the map
$p_n\mapsto n$ respects the group structure of the integers, and we conclude that
$p_n$ represents the $K_0$ class labeled by $n$. In this sense, one
may consider the projection $p_1$ to be a generator of $K_0$.

\section{Bimodules}\label{sec:bimodules}

\noindent
We now construct $\Chinf$-modules on the space of Schwartz functions
in one discrete and one real variable in analogy with the
noncommutative torus. We show that one may construct left and right
$\Chinf$-modules, as well as bimodules, depending on a set of
parameters. Furthermore, it turns out that these modules are in fact
isomorphic to a number of copies of the algebra itself.  To begin
with, for $\xi,\eta\in\mathcal{S}(\reals\times\integers)$ set
\begin{align*}
  &(\xi,\eta)_L = \sum_{k\in\integers}\int_{\reals}
  \xi(x,k)\overline{\eta(x,k)}dx\\
  &(\xi,\eta)_R = \sum_{k\in\integers}\int_{\reals}
    \overline{\xi(x,k)}\eta(x,k)dx
        =\overline{(\xi,\eta)}_L .
\end{align*}
The corresponding left module structure is given in the following
result.

\begin{proposition}\label{prop:left.module}
  Let $\lambda_0,\lambda_1,\eps,\hbar\in\reals$ and $r\in\integers$ be such
  that $\lambda_0\eps+\lambda_1r = -\hbar$. \\ For $f=\sum_{n\in\integers}f_n(u)W^n$
  set
  \begin{align}\label{eq:Ch.left.module.action}
    (f \xi) (x,k)  =\sum_{n\in\integers}f_n(\lambda_0x+\lambda_1k)\xi(x-n\eps,k-nr)
  \end{align}
  for $\xi\in\mathcal{S}(\reals\times\integers)$. Then
  $\mathcal{S}(\reals\times\integers)$ is a left $\Chinf$-module such
  that
  \begin{align*}
    (f\xi,\eta)_L = (\xi,f^\ast\eta)_L
  \end{align*}
  for all $f\in\Chinf$.
\end{proposition}

\begin{proof}
  In order for \eqref{eq:Ch.left.module.action} to define a module
  action, one has to check that it respects the relations in the
  algebra; i.e. $\paraa{(fg)\xi}(x,k)=\paraa{f(g\xi)}(x,k)$. One
  finds that
  \begin{align*}
	  &\quad\  \paraa{f(g\xi)}(x,k)
	     = \sum_{n\integers}f_n(\lambda_0x+\lambda_1k)(g\xi)(x-n\eps,k-nr)\\
           &=\!\sum_{n,m\in\integers}\!f_n(\lambda_0x+\lambda_1k)
             g_m\paraa{\lambda_0(x-n\eps)+\lambda_1(k-nr)}\\
	     &\qquad\quad\times\xi\paraa{x-(n+m)\eps,k-(n+m)r}\\
           &=\!\sum_{n,l\in\integers}\!f_n(\lambda_0x+\lambda_1k)
             g_{l-n}\paraa{\lambda_0x+\lambda_1k-n(\lambda_0\eps+\lambda_1r)}\xi\paraa{x-l\eps,k-lr}\\
           &=\!\sum_{n,l\in\integers}\!f_n(\lambda_0x+\lambda_1k)
             g_{l-n}\paraa{\lambda_0x+\lambda_1k+n\hbar}\xi\paraa{x-l\eps,k-lr},
  \end{align*}
  by using that $\lambda_0\eps+\lambda_1r=-\hbar$.  On the other hand
  \begin{align*}
	  &\quad\ \paraa{(fg)\xi}(x,k)
    =\parac{\sum_{n,m\in\integers}f_m(u)g_{n-m}(u+m\hbar)W^n\xi}(x,k)\\
    &=\sum_{n,m\in\integers}f_m(\lambda_0x+\lambda_1k)g_{n-m}(\lambda_0x+\lambda_1k+m\hbar)\xi(x-m\eps,k-mr),
  \end{align*}
  which is seen to equal $\paraa{f(g\xi)}(x,k)$. Next, let us show that $(f\xi,\eta)_L=(\xi,f^\ast\eta)_L$ by first computing
  \begin{align*}
    (f\xi,\eta)_L
    &= \sum_{k\in\integers}\int_{\reals}\parac{\sum_{n\in\integers}f_n(u)W^n\xi}(x,k)\overline{\eta(x,k)}dx\\
    &= \sum_{k,n\in\integers}\int_{\reals}f_n(\lambda_0x+\lambda_1k)\xi(x-n\eps,k-nr)\overline{\eta(x,k)}dx.
  \end{align*}
  Let us compare this with
  \begin{align*}
    (\xi,f^\ast \eta)_L
    &= \sum_{k\in\integers}\int_\reals\xi(x,k)\overline{\parac{\sum_{n\in\integers}\overline{f_{-n}(u+n\hbar)}W^n\eta}(x,k)}dx\\
    &= \sum_{k,n\in\integers}\int_\reals\xi(x,k)f_{-n}(\lambda_0x+\lambda_1k+n\hbar)\overline{\eta(x-n\eps,k-nr)}dx.
  \end{align*}
  Setting $m=-n$ gives
  \begin{align*}
    (\xi,f^\ast \eta)_L
    &= \sum_{k,m\in\integers}\int_\reals\xi(x,k)f_{m}(\lambda_0x+\lambda_1k-m\hbar)\overline{\eta(x+m\eps,k+mr)}dx
  \end{align*}
  and changing the integration variable to $y=x+m\eps$ yields
  \begin{align*}
    (\xi,f^\ast \eta)_L
    &=\!\! \sum_{k,m\in\integers}\int_\reals\xi(y\!-\!m\eps,k)f_{m}(\lambda_0y\!+\!\lambda_1k\!-\!\lambda_0m\eps\!-\!m\hbar)\overline{\eta(y,k\!+\!mr)}dy.    
  \end{align*}
  Finally, we set $l=k+mr$ and use that $\lambda_0\eps+\lambda_1r=-\hbar$ to obtain
  \begin{align*}
    (\xi,f^\ast \eta)_L
    &= \sum_{l,m\in\integers}\int_\reals\xi(y-m\eps,l-nr)f_{m}(\lambda_0y+\lambda_1l)\overline{\eta(y,l)}dy,    
  \end{align*}
  which equals $(f\xi,\eta)_L$.
\end{proof}

\noindent
In the same way, one may construct a right module structure on
$\mathcal{S}(\reals\times\integers)$. The proof is analogous to that of
Proposition~\ref{prop:left.module}.

\begin{proposition}\label{prop:right.module}
  Let $\mu_0,\mu_1,\eps',\hbar'\in\reals$ and $r'\in\integers$ such
  that $\mu_0\eps'+\mu_1r' = \hbar'$.  \\
  For
  $f=\sum_{n\in\integers}f_n(u)W^n$ set
  \begin{align}\label{eq:Ch.right.module.action}
    (\xi f)(x,k)
    =\sum_{n\in\integers}f_n(\mu_0x+\mu_1k-n\hbar')\xi(x-n\eps',k-nr')
  \end{align}
  for $\xi\in\mathcal{S}(\reals\times\integers)$. Then
  $\mathcal{S}(\reals\times\integers)$ is a right $\Chpinf$-module
  such that
  \begin{align*}
    (\xi f,\eta)_R = (\xi,\eta f^\ast)_R
  \end{align*}
  for all $f\in\Chpinf$.
\end{proposition}

\noindent 
A (left or right) $\Chinf$-module constructed as above will be denoted
by $\E_{\hbar}$ with a suitable choice of parameters
  implicitly assumed. If the parameters of the left and
right module structures are compatible, $\mathcal{S}(\reals\times\integers)$
becomes a $(\Chinf,\Chpinf)$-bimodule.

\begin{proposition}\label{prop:bimodule}
  Let
  $\lambda_0,\mu_0,\lambda_1,\mu_1,\eps,\eps',\hbar,\hbar'\in\reals$
  and $r,r'\in\integers$ such that
  \begin{alignat*}{2}
    &\lambda_0\eps + \lambda_1r = -\hbar &\qquad
    &\mu_0\eps'+\mu_1r' = \hbar'\\
    &\lambda_0\eps' + \lambda_1 r' = 0 &
    &\mu_0\eps + \mu_1r = 0.
  \end{alignat*}
  Then $\mathcal{S}(\reals\times\integers)$ is a
  $(\Chinf,\Chpinf)$-bimodule with respect to the left and right
  actions defined in Proposition~\ref{prop:left.module} and
  Proposition~\ref{prop:right.module}.
\end{proposition}

\begin{proof}
  In order for $\mathcal{S}(\reals\times\integers)$ to be a
  $(\Chinf,\Chpinf)$-bimodule, it must hold that
  \begin{align*}
    \paraa{g(\xi f)}(x,k) = \paraa{(g\xi)f}(x,k)
  \end{align*}
  for all $\xi\in\mathcal{S}(\reals\times\integers)$, $g\in\Chinf$ and
  $f\in\Chpinf$. One finds that
  \begin{align*}
	  &\quad\ \paraa{g(\xi f)}(x,k)
    = \sum_{n\in\integers}g_n(\lambda_0x+\lambda_1k)(\xi f)(x-n\eps,k-nr)\\
    &= \sum_{n,m\in\integers}g_n(\lambda_0x+\lambda_1k)
      f_m\paraa{\mu_0(x-n\eps)+\mu_1(k-nr)-m\hbar'}\\
    &\qquad\qquad\times \xi(x-n\eps-m\eps',k-nr-mr')\\
    &=\sum_{n,m\in\integers}g_n(\lambda_0x+\lambda_1k)
      f_m(\mu_0x+\mu_1k-m\hbar')\xi(x-n\eps-m\eps',k-nr-mr'),
  \end{align*}
  by using that $\mu_0\eps+\mu_1r = 0$. On the other hand
  \begin{align*}
    \paraa{(g\xi)f}(x,k)
	     & = \sum_{m\in\integers}f_m(\mu_0x+\mu_1k-m\hbar')
              (g\xi)(x-m\eps',k-mr')\\
            &=\sum_{m,n\in\integers}f_m(\mu_0x+\mu_1k-m\hbar')
              g_n\paraa{\lambda_0(x-m\eps')+\lambda_1(k-mr')}\\
            &\qquad\qquad\times\xi(x-m\eps'-n\eps,k-mr'-nr)\\
            &=\sum_{m,n\in\integers}g_n\paraa{\lambda_0x+\lambda_1k}
              f_m(\mu_0x+\mu_1k-m\hbar')\\
	      &\qquad\qquad\times\xi(x-m\eps'-n\eps,k-mr'-nr),
  \end{align*}
  since $\lambda_0\eps'+\lambda_1r'=0$. We conclude that
  $\paraa{g(\xi f)}(x,k) = \paraa{(g\xi)f}(x,k)$.
\end{proof}

\noindent
A bimodule defined as in Proposition~\ref{prop:bimodule} will be
denoted by $\Ehhp$, again tacitly assuming a choice of the parameters
$\lambda_0,\lambda_1,\mu_0,\mu_1,\eps,\eps'\in\reals$ and
$r,r'\in\integers$ satisfying the requirements in
Proposition~\ref{prop:bimodule}.  Note that one can construct
bimodules for arbitrary choices of $\hbar,\hbar'$ by choosing
e.g. $\eps,\eps',r,r'$ such that $\eps r'\neq \eps'r$ and setting
\begin{alignat*}{2}
  &\lambda_0 = -\frac{\hbar r'}{\eps r'-\eps'r}&\quad
  &\lambda_1 = \frac{\hbar\eps'}{\eps r'-\eps'r}\\
  &\mu_0= -\frac{\hbar'r}{\eps r'-\eps'r} &
  &\mu_1 = \frac{\hbar'\eps}{\eps r'-\eps'r}.
\end{alignat*}

\noindent
Let us now point out some obvious isomorphisms between modules defined
by different sets of parameters. To simplify the description, we make
the following definition.


\begin{definition}
  The vectors $(x_1,\ldots,x_n)\in\reals^n$ and
  $(y_1,\ldots,y_n)\in\reals^n$ are called \emph{$0$-compatible} if
  either $x_i=y_i=0$ or $x_i,y_i\neq 0$ for $i=1,\ldots,n$.
\end{definition}


\begin{proposition}\label{prop:module.param.iso}
  Let $\E_{\hbar}$ and $\tilde{\E}_{\hbar}$ be left $\Chinf$-modules
  (as in Proposition~\ref{prop:left.module}) defined by the parameters
  $(\lambda_0,\eps,\lambda_1,r)$ and
  $(\lambdat_0,\epst,\lambda_1,r)$, respectively. If
  $(\lambda_0,\eps)$ and
  $(\lambdat_0,\epst)$ are $0$-compatible then
  $\E_{\hbar}\simeq\tilde{\E}_{\hbar}$ as left $\Chinf$-modules.
\end{proposition}

\begin{proof}
  Note that since
  $\lambda_0\eps+\lambda_1r=-\hbar=\lambdat_0\epst+\lambda_1r$, it
  follows that $\lambda_0\eps=\lambdat_0\epst$.  We shall proceed by
  defining module homomorphisms
  $\phi_{\tau}:\E_{\hbar}\to\tilde{\E}_{\hbar}$ as
  \begin{align*}
    \phi_{\tau}(\xi)(x,k) = \xi(\tau x,k)
  \end{align*}
  for $\tau\in\reals$. Note that $\phi_{\tau}$
  is a linear map and if $\tau\neq 0$ then $\phi_{\tau}$ is
  invertible. Now, let us derive conditions for $\phi_{\tau}$ to be
  a module homomorphism; thus, we demand that
  $\phi_{\tau}(f\xi)=f\phi_{\tau}(\xi)$ for all $f\in\Chinf$. To
  this end one computes
  \begin{align*}
    \phi_{\tau}(f\xi)(x,k)
    &= \sum_{n\in\integers}f_n(\lambda_0\tau x+\lambda_1k)\xi(\tau x-n\eps,k-nr)\\
    \paraa{f\phi_{\tau}(\xi)}(x,k)
    &=\sum_{n\in\integers}f_n(\lambdat_0x+\lambda_1k)\phi_{\tau}(\xi)(x-n\epst,
      k-nr)\\
    &=\sum_{n\in\integers}f_n(\lambdat_0x+\lambda_1k)\xi\paraa{\tau(x-n\epst),
      k-nr}.
  \end{align*}
  The above expressions are equal if
  \begin{align}\label{eq:lambdatildetaum}
    \lambda_0\tau = \lambdat_0\quad\text{and}\quad
    \eps=\tau\epst.
  \end{align}
  Note that since $(\lambda_0,\eps)$ and
  $(\lambdat_0,\epst)$ are $0$-compatible, either both
  sides of each equation are zero (i.e. trivially giving a solution)
  or both sides are non-zero (as long as $\tau\neq 0$). Thus, if
  $\lambda_0=\lambdat_0=0$ and $\eps=\epst=0$, then $\phi_{\tau}$ is an
  isomorphism for any $\tau\neq 0$. If $\lambda_0,\lambdat_0\neq 0$
  and $\eps=\epst=0$, one can set $\tau=\lambdat_0/\lambda_0$, solving
  \eqref{eq:lambdatildetaum} (and similarly in the case when
  $\eps,\epst\neq 0$ but $\lambda_0=\lambdat_0=0$). Now, for the case
  when $\lambda_0,\lambdat_0,\eps,\epst\neq 0$ one sets
  $\tau=\lambdat_0/\lambda_0$ and notes that
  \begin{align*}
    \tau\epst = \frac{\lambdat_0\epst}{\lambda_0}
    =\frac{\lambda_0\eps}{\lambda_0}=\eps
  \end{align*}
  giving a solution of \eqref{eq:lambdatildetaum}. Thus, it follows
  that under the assumptions given in the proposition, there exists a
  module isomorphism $\phi_\tau:\E_{\hbar}\to\tilde{\E}_{\hbar}$.
\end{proof}

\noindent
Somewhat surprisingly, it turns out that these modules are in fact isomorphic (as modules) to copies of the algebra
itself. More precisely, we formulate the statement as follows.

\begin{proposition}\label{prop:module.iso.r.algebra}
  Let $\E_\hbar$ be a left $\Chinf$-module defined by the parameters
  $\lambda_0,\lambda_1,\eps,r$ such that
  $\lambda_0\eps+\lambda_1r=-\hbar$ and $\lambda_0\neq 0$. For
  arbitrary $F\in(\Chinf)^r$ (considered as a left $\Chinf$-module) we
  write $F = (F^0,F^1,\ldots,F^{r-1})$ with $F^k\in\Chinf$ for
  $k=0,1,\ldots,r-1$, and introduce the components $F^k_n(u)$ via
  \begin{align*}
    F^k = \sum_{n\in\integers}F^k_n(u)W^n.
  \end{align*}
  Furthermore, for an integer $k$ we let $k_0\in\integers$ and
  $0\leq k_1\leq r-1$ be defined by $k=k_0r+k_1$.  Then the map
  $\phi:\paraa{\Chinf}^r\to\E_{\hbar}$, defined as
  \begin{align*}
    \phi(F)(x,k) = F^{k_1}_{k_0}(\lambda_0x+\lambda_1k),
  \end{align*}
  is an isomorphism of left $\Chinf$-modules.
\end{proposition}

\begin{proof}
  First of all, it is clear that
  $\phi(F+F')=\phi(F)+\phi(F')$. Furthermore, one finds that
  \begin{align*}
    \phi(fF)(x,k)
    &= (fF)^{k_1}_{k_0}(\lambda_0x+\lambda_1k)\\
    &= \sum_{n\in\integers}f_n(\lambda_0x+\lambda_1k)F^{k_1}_{k_0-n}(\lambda_0x+\lambda_1k+n\hbar)
  \end{align*}
  as well as
  \begin{align*}
    \paraa{f\phi(F)}(x,k)
    &= \sum_{n\in\integers}f(\lambda_0x+\lambda_1k)\phi(F)(x-n\eps,k-nr)\\
    &= \sum_{n\in\integers}f(\lambda_0x+\lambda_1k)\phi(F)(x-n\eps,(k_0-n)r+k_1)\\    
    &= \sum_{n\in\integers}f(\lambda_0x\!+\!\lambda_1k)
      F^{k_1}_{k_0-n}\paraa{\lambda_0(x\!-\!n\eps)\!+\!\lambda_1(k_0r\!+\!k_1)\!-\!n\lambda_1r}\\
    &= \sum_{n\in\integers}f(\lambda_0x+\lambda_1k)
      F^{k_1}_{k_0-n}\paraa{\lambda_0x+\lambda_1k+n\hbar}
      =\phi(fF)(x,k)
  \end{align*}
  by using that $-n(\lambda_0\eps+\lambda_1r)=n\hbar$. Hence, $\phi$
  is a left module homomorphism. One may readily construct the inverse
  \begin{align*}
    \phi^{-1}(\xi)=(F^0,\ldots,F^{r-1}):\ \ 
    F^k = \sum_{n\in\integers}\xi\paraa{\tfrac{1}{\lambda_0}(u-\lambda_1(nr+k)),nr+k}W^n
  \end{align*}
  and check that
  \begin{align*}
    \phi\paraa{\phi^{-1}(\xi)}(x,k)
    &= \paraa{\phi^{-1}(\xi)}^{k_1}_{k_0}(\lambda_0x+\lambda_1k)\\
    &= \xi\paraa{x-\lambda_1(k_0r+k_1)+\lambda_1k,k_0r+k_1}=\xi(x,k).
  \end{align*}
  We conclude that $\phi$ is indeed a left module isomorphism.
\end{proof}

\noindent
In the case when $r=1$ and, moreover, $\E_{\hbar}$ is a $\Chinf$-bimodule, one can
strengthen the result to obtain a bimodule isomorphism.

\begin{proposition}
  For $\hbar>0$, let $\E_{\hbar,\hbar}$ be a
  $\Chinf$-bimodule (as in Proposition~\ref{prop:bimodule})
  with $\lambda_0\neq 0$ and $r=r'=1$. The map
  $\phi:\Chinf\to\E_{\hbar,\hbar}$, defined as
  \begin{align}
    \phi(f)(x,k) = f_k\paraa{\lambda_0x+\lambda_1k}\label{eq:iso.E.r.one}
  \end{align}
  for $f=\sum_{n\in\integers}f_n(u)W^n$, is a bimodule isomorphism
  with inverse
  \begin{align*}
    \phi^{-1}(\xi) = \sum_{n\in\integers}
    \xi\paraa{\tfrac{1}{\lambda_0}(u-\lambda_1n),n}W^n.    
  \end{align*}
  for $\xi\in\E_{\hbar,\hbar}$.
\end{proposition}

\begin{proof}
  The fact that $\phi$ is a left module isomorphism with 
  \begin{align*}
    \phi^{-1}(\xi) = \sum_{n\in\integers}
    \xi\paraa{\tfrac{1}{\lambda_0}(u-\lambda_1n),n}W^n
  \end{align*}
  follows immediately from Proposition~\ref{prop:module.iso.r.algebra}
  (and its proof). Before showing that $\phi$ is a also a right
  $\Chinf$-module homomorphism, let us derive a few properties of the
  parameters defining the module. For a bimodule with $\hbar=\hbar'>0$
  and $r=r'=1$ one necessarily has $\eps\neq\eps'$ (otherwise implying
  $\hbar=0$). Hence, one may solve for $\lambda_0,\mu_0$ and
  $\lambda_1,\mu_1$ to obtain
  \begin{align*}
    \lambda_0=\mu_0 = \frac{\hbar}{\eps'-\eps}\quad
    \lambda_1 = -\frac{\hbar\eps'}{\eps'-\eps}\qquad
    \mu_1 = -\frac{\hbar\eps}{\eps'-\eps}
  \end{align*}
  and we note that
  \begin{align}\label{eq:mu.m.hbar}
    \mu_1-\hbar = \frac{-\hbar\eps-(\eps'-\eps)\hbar}{\eps'-\eps}
    =-\frac{\hbar\eps'}{\eps'-\eps}=\lambda_1.
  \end{align}
  Now, one computes for $f,g\in\Chinf$
  \begin{align*}
    \phi(fg)(x,k)
    &=(fg)_k(\lambda_0x+\lambda_1k)\\
    &=\sum_{n\in\integers}f_n(\lambda_0x+\lambda_1k)g_{k-n}(\lambda_0x+\lambda_1k+n\hbar)                
  \end{align*}
  and
  \begin{align*}
    \paraa{\phi(f)g}(x,k)
    &= \sum_{l\in\integers}g_l(\mu_0x+\mu_1k-l\hbar)\phi(f)\paraa{x-l\eps',k-l}\\
    &=\sum_{l\in\integers}f_{k-l}\paraa{\lambda_0(x-l\eps')+\lambda_1(k-l)}g_l(\mu_0x+\mu_1k-l\hbar)\\
    &=\sum_{n\in\integers}f_{k-l}\paraa{\lambda_0x+\lambda_1k-l(\lambda_0\eps'+\lambda_1)}g_l(\mu_0x+\mu_1k-l\hbar)\\
    &=\sum_{n\in\integers}f_{k-l}\paraa{\lambda_0x+\lambda_1k}g_l(\mu_0x+\mu_1k-l\hbar)
  \end{align*}
  since
  $0=\lambda_0\eps'+\lambda_1r'=\lambda_0\eps'+\lambda_1$. Moreover,
  changing the summation index to $n=k-l$ gives
  \begin{align*}
    \paraa{\phi(f)g}(x,k)
    &=\sum_{n\in\integers}f_n\paraa{\lambda_0x+\lambda_1k}g_{k-n}(\mu_0x+(\mu_1-\hbar)k+n\hbar)\\
    &=\sum_{n\in\integers}f_n\paraa{\lambda_0x+\lambda_1k}g_{k-n}(\lambda_0x+\lambda_1k+n\hbar)
      =\phi(fg)(x,k),
  \end{align*}
  by using that $\lambda_0=\mu_0$ and $\mu_1-\hbar=\lambda_1$, as shown
  in \eqref{eq:mu.m.hbar}.  
\end{proof}




\subsection{Hermitian structures}
\noindent
Continuing the analogy with the noncommutative torus, we show 
there exist hermitian structures on the
$(\Ch,\C_{\hbar'})$-bimodule $\Ehhp$.

\begin{proposition}\label{prop:hermitian.structure.bimodule}
  Let $\Ehhp$ be a $(\Ch,\C_{\hbar'})$-bimodule, as in
  Proposition~\ref{prop:bimodule}, and define
  $\lhp{\cdot,\cdot}:\Ehhp\times\Ehhp\to\Ch$ and
  $\rhp{\cdot,\cdot}:\Ehhp\times\Ehhp\to\C_{\hbar'}$ as
  \begin{align*}
    &\lhp{\xi,\eta} = \sum_{n\in\integers}\bracketc{\int_\reals
    \paraa{\xi,e^{2\pi i\xh u}W^n\eta}_Le^{2\pi i\xh u}d\xh}W^n\\
    &\rhp{\xi,\eta} = \sum_{n\in\integers}\bracketc{\int_\reals
    \paraa{\xi e^{2\pi i\xh u}W^n,\eta}_Re^{2\pi i\xh u}d\xh}W^n.
  \end{align*}
  Then it follows that
  \begin{align*}
    &\lhp{a\xi,\eta} = a\paraa{\lhp{\xi,\eta}}\qquad
      \rhp{\xi,\eta b} = \paraa{\rhp{\xi,\eta}}b
  \end{align*}
  as well as the compatibility condition
  \begin{align*}
    \lhp{\xi,\eta}\psi = \xi\rhp{\eta,\psi}
  \end{align*} 
  for $a\in\Ch$, $b\in\C_{\hbar'}$ and $\xi,\eta,\psi\in\Ehhp$.
\end{proposition}

\begin{remark}
  Strictly speaking, $e^{2\pi i\xh u}W^n\eta$ is not defined since
  $e^{2\pi i\xh u}$ does not decay as $u\to\infty$ (and, hence, does
  not belong to the algebra). However, having in mind the left action
  \eqref{eq:Ch.left.module.action}, we interpret the above expression
  as
  \begin{align*}
    (\xi,e^{2\pi i\xh u}W^n\eta)_L = 
    \sum_{k\in\integers}\int_{\reals}
    \xi(x,k)\overline{\eta(x-n\eps,k-nr)}e^{-2\pi i\xh(\lambda_0x+\lambda_1k)}dx
  \end{align*}
  and allow ourselves a formulation as in
  Proposition~\ref{prop:hermitian.structure.bimodule} since it more
  clearly reflects the idea behind the construction. Similarly for the
  right structure.
\end{remark}

\begin{proof}
  Let us start by showing that
  $\lhp{a\xi,\eta}=a (\lhp{\xi,\eta})$. Thus, we set
  \begin{align*}
    a = \sum_{k}a_k(u)W^k
  \end{align*}
  and write
  \begin{align*}
    \lhp{a\xi,\eta}
    &=
      \sum_{k,n}\int_\reals
      \paraa{a_k(u)W^k\xi,e^{2\pi i\xh u}W^n\eta}_Le^{2\pi i\xh u}W^nd\xh\\
    &=\sum_{k,n}\int_\reals\paraa{\xi,W^{-k}\overline{a_k(u)}e^{2\pi i\xh u}W^n\eta}_Le^{2\pi i\xh u}W^nd\xh\\
    &=\sum_{k,n}\int_\reals\paraa{\xi,\overline{a_k(u-k\hbar)}e^{2\pi i\xh(u-k\hbar)}W^{n-k}\eta}_Le^{2\pi i\xh u}W^nd\xh.\\
    &=\sum_{k,n}\int_\reals\paraa{\xi,\overline{a_k(u-k\hbar)}e^{2\pi i\xh}W^{n-k}\eta}_Le^{2\pi i\xh k\hbar}e^{2\pi i\xh u}W^nd\xh.
  \end{align*}
  Now, let us replace $a_k(u-k\hbar)$ by its Fourier integral:
  \begin{align*}
    \lhp{a\xi,\eta}
    &=\!\sum_{k,n}\iint_{\reals^2}\!\paraa{\xi,\overline{\hat{a}_k(x)}e^{-2\pi ix(u-k\hbar)}e^{2\pi i\xh u}W^{n-k}\eta}_Le^{2\pi i\xh k\hbar}e^{2\pi i\xh u}W^ndx d\xh\\
    &=\!\sum_{k,n}\iint_{\reals^2}\paraa{\xi,e^{2\pi iu(\xh-x)}W^{n-k}\eta}_L\hat{a}_k(x)
      e^{2\pi i (\xh-x)k\hbar}e^{2\pi i\xh u}W^ndx d\xh.
  \end{align*}
  By a change of variables we set $l=n-k$ and $y=\xh-x$, giving
  \begin{align*}
    \lhp{a\xi,\eta}
    &=\sum_{k,l}\iint_{\reals^2}\paraa{\xi,e^{2\pi iuy}W^l\eta}_L\hat{a}_k(x)
      e^{2\pi iyk\hbar}e^{2\pi i(x+y) u}W^{k+l}dx dy\\
    &= \sum_{k,l}\int_\reals\hat{a}_k(x)e^{2\pi ixu} dx\int_{\reals^2}\paraa{\xi,e^{2\pi iuy}W^l\eta}_L
      e^{2\pi i yk\hbar}e^{2\pi i yu}W^{k+l}dy\\
    &=\sum_{k,l} a_k(u)W^k\int_\reals\paraa{\xi,e^{2\pi iuy}W^l\eta}_Le^{2\pi iyu}W^l dy
      =a\paraa{\lhp{\xi,\eta}}.
  \end{align*}
  The proof of the statement $\rhp{\xi,\eta a}=(\rhp{\xi,\eta})a$ is
  completely analogous. Finally, we need to show compatibility of the
  two products; namely, that
  \begin{align*}
    \lhp{\xi,\eta}\psi = \xi\rhp{\eta,\psi}
  \end{align*}
  for $\xi,\eta,\psi\in\Ehhp$. One writes
  \begin{align*}
    (\lhp{\xi,\eta}\psi)(x,k)
	  &=\sum_{n}\int_\reals\paraa{\xi,e^{2\pi i\xh u}W^n\eta}_L
      e^{2\pi i\xh(\lambda_0x+\lambda_1k)}\psi(x-n\eps,k-nr)d\xh\\
    &=\sum_{n,l}\iint_{\reals^2}\xi(y,l)e^{-2\pi i\xh(\lambda_0y+\lambda_1l)}\overline{\eta(y-n\eps,l-nr)}\\
	  &\qquad\quad\times e^{2\pi i\xh(\lambda_0x+\lambda_1k)}\psi(x-n\eps,k-nr)dy d\xh\\
    &= \sum_{n,l}\int_\reals\int_\reals e^{2\pi i\xh(\lambda_0(x-y)+\lambda_1(k-l))}\\
	  &\qquad\quad\times\xi(y,l)\overline{\eta(y-n\eps,l-nr)}\psi(x-n\eps,k-nr)dy d\xh\\
    &=\sum_{n,l}\int_{\reals}\delta(\lambda_0(x-y)+\lambda_1(k-l))
      \xi(y,l)\\
	  &\qquad\quad\times\overline{\eta(y-n\eps,l-nr)}\psi(x-n\eps,k-nr)dy\\
    &=\sum_{n,l}\xi\paraa{x+(k-l)\lambda_1/\lambda_0,l}\\
      &\qquad\quad\times\overline{\eta\paraa{x+(k-l)\lambda_1/\lambda_0-n\eps,l-nr}}\psi(x-n\eps,k-nr).
  \end{align*}
  A similar computation for $\xi\rhp{\eta,\psi}$ gives
  \begin{align*}
	  \paraa{\xi\rhp{\eta,\psi}}(x,k) =& \sum_{n',l'}\xi(x-n'\eps',k-n'r')\\[-1ex]
	  &\qquad\times\overline{\eta\paraa{x+(k-l')\mu_1/\mu_0-n'\eps',l'-n'r'}}\\
    &\qquad\times  \psi\paraa{x+(k-l')\mu_1/\mu_0,l'}
  \end{align*}
  To prove that the expressions for $(\lhp{\xi,\eta}\psi)(x,k)$ and
  $\paraa{\xi\rhp{\eta,\psi}}(x,k)$ are equal for all $x$ and $k$, we
  compare the sums term by term. Let us proceed as follows: We fix
  arbitrary $n$ in the expression for $(\lhp{\xi,\eta}\psi)(x,k)$ and
  arbitrary $n'$ in the expression for
  $\paraa{\xi\rhp{\eta,\psi}}(x,k)$. Now, for every such choice we
  will prove that there exists $l$ and $l'$ such that the
  corresponding terms are equal. Setting
  \begin{align*}
    l = k-n'r'\quad\text{and}\quad l'=k-nr,
  \end{align*}
  one finds (by comparing the arguments of $\xi$, $\eta$ and $\psi$)
  that the corresponding terms in the two sums are equal if
  \begin{align*}
    &\frac{\lambda_1}{\lambda_0}(k-l) = -n'\eps'\\
    &\frac{\lambda_1}{\lambda_0}(k-l)-n\eps =
      \frac{\mu_1}{\mu_0}(k-l')-n'\eps'\\
    &l-nr = l'-n'r'\\
    &-n\eps = \frac{\mu_1}{\mu_0}(k-l')
  \end{align*}
  Inserting $l=k-n'r'$ and $l'=k-nr$ into these equations yields
  \begin{align*}
    \lambda_0\eps'+\lambda_1r'=0\quad\text{and}\quad
    \mu_0\eps+\mu_1r = 0
  \end{align*}
  as the remaining conditions. However, these conditions are true due
  to the fact that $\Ehhp$ is assumed to be a bimodule fulfilling the
  requirements of Proposition~\ref{prop:bimodule}.  
\end{proof}

\subsection{Bimodule connections}

Let us now turn to the question of finding connections $\nabla$ on the
bimodule $\Ehhp$, of the form
\begin{align}\label{eq:bimodule.conn.def}
  \nabla_1\xi(x,k) = \alpha\xi'_x(x,k)\qquad
  \nabla_2\xi(x,k) = \beta x\xi(x,k) + \gamma k\xi(x,k),
\end{align}
for $\alpha,\beta,\gamma\in\complex$, and we start by working out the
conditions for $\nabla$ to be a left connection; that is, such that
\begin{align*}
\nabla_k(f\xi) = f\nabla_k\xi + (\partial_kf)\xi
\end{align*}
for $k=1,2$, $\xi\in\Ehhp$ and $f\in\Chinf$. 

\begin{proposition}\label{prop:left.module.connection}
  Let $\E_\hbar$ be a left $\Chinf$-module with
  respect to a choice of $\lambda_0,\lambda_1,\eps,r$ (as in
  Proposition~\ref{prop:left.module}) with $\lambda_0\neq 0$. If
  $\alpha=1/\lambda_0$ and $\beta\eps+\gamma r=2\pi i$ then $\nabla$,
  as defined in \eqref{eq:bimodule.conn.def}, is a left module
  connection.
\end{proposition}

\begin{proof}
  One finds that
  \begin{align*}
    \nabla_1(f\xi)(x,k)
    &= \alpha\paraa{f\xi'_x}(x,k)+\alpha\lambda_0\paraa{(\partial_1 f)\xi}(x,k)\\
    &= (f\nabla_1\xi)(x,k)+\paraa{(\partial_1f)\xi}(x,k)
  \end{align*}
  since $\alpha\lambda_0=1$. Furthermore,
  \begin{align*}
    \nabla_2(f\xi)(x,k)
    &= \paraa{f(\nabla_2\xi)}(x,k)
      +\frac{\beta\eps+\gamma r}{2\pi i}\paraa{(\partial_2f)\xi}(x,k)\\
    &= \paraa{f(\nabla_2\xi)}(x,k)+\paraa{(\partial_2f)\xi}(x,k)
  \end{align*}
  since $\beta\eps+\gamma r = 2\pi i$, showing that $\nabla$ is a left
  module connection.
\end{proof}

\begin{remark}
By using the isomorphism
$\phi:\Chinf\to\E_{\hbar}$ in \eqref{eq:iso.E.r.one} one may induce
connections on the algebra itself via
\begin{align*}
  \nablat_kf = \phi^{-1}\paraa{\nabla_k\phi(f)}
\end{align*}
for $f\in\Chinf$, giving
\begin{align*}
\nablat_1f = \partial_1f 
\end{align*}
and
\begin{align*}
  \nablat_2f = \parac{\frac{\alpha\beta \hbar}{2 \pi i } + 1} \, \partial_2 f + \alpha\beta \, u f = \partial_2f + \alpha\beta fu \, ,
\end{align*}
using the expression \eqref{der-com} for $\partial_2$.
\end{remark}

It is straightforward to compute the curvature of $\nabla$ in \eqref{eq:bimodule.conn.def}:
\begin{align*}
  F_{12} \xi(x,k) = (\nabla_1\nabla_2 - \nabla_2\nabla_1) \xi(x,k)=\alpha\beta \, \xi(x,k)
\end{align*}
showing that $\nabla$ has a constant curvature equal to $\alpha\beta$. Note
that the curvature does not depend on $\gamma$, which implies that one
may construct connections with arbitrary constant curvature. Namely,
let $R\in\complex$ and set
\begin{align*}
  \alpha=\frac{1}{\lambda_0}\qquad
  \beta = \lambda_0R\qquad
  \gamma = \frac{2\pi i-\lambda_0\eps R}{r}
\end{align*}
clearly fulfilling the requirements of
Proposition~\ref{prop:left.module.connection}, giving a connection of
constant curvature $\alpha\beta=R$.

Moreover, one easily shows that the curvature is $\Chinf$-linear, that is
\begin{align*}
  F_{12} (f \xi(x,k)) = f (F_{12} \xi(x,k))  
\end{align*}
for $\xi\in\E_\hbar$ and $f\in\Chinf$, implying that
the curvature $F_{12}$ is an element of
\begin{align*}
  \textup{End}_{\Chinf}(\E_{\hbar}) =
  \{ T : \E_{\hbar} \to \E_{\hbar})
  ~|~ T(f \xi(x,k)) = f (T \xi(x,k)) \}
\end{align*}
for $\xi\in\E_\hbar$ and $f\in\Chinf$; this is an algebra under composition. 

Given a left connection on $\E_\hbar$, one can define
associated derivations on
$\textup{End}_{\Chinf}(\E_\hbar)$ from the commutators
\begin{align}\label{der-conn}
\delta_k(T) = \nabla_k \circ T - T \circ \nabla_k
\end{align}
for $k=1,2$ and
$T \in \textup{End}_{\Chinf}(\E_\hbar)$.  If
$\E_\hbar$ is also a right $\Chpinf$-module, then
$\Chpinf$ (acting on the right) is a subalgebra of
$\textup{End}_{\Chinf}(\E_\hbar)$. On the algebra $\Chpinf$, with
the connection $\nabla$ as in \eqref{eq:bimodule.conn.def} one
recovers a rescaling of the natural derivations from \eqref{der-conn},
as stated in the following result.

\begin{proposition}\label{prop:right.module.derivation}
  Let $\E_{\hbar,\hbar'}$ be a $(\Chinf,\Chpinf)$-bimodule with and
  let $\nabla$ be a left module connection as defined in
  \eqref{eq:bimodule.conn.def}. Then on $f\in\Chpinf$ the derivations
  in \eqref{der-conn} are given by
  \begin{align*}
    \delta_1(f) = (\alpha \mu_0) \partial_1 f
    \qquad  \delta_2(f) = \frac{\beta \eps'+\gamma r'}{2\pi i}\partial_2 f.    
  \end{align*}
\end{proposition}

\begin{proof}
  With the right structure as in \eqref{eq:Ch.right.module.action}, one computes
    \begin{align*}
   (\xi \delta_1 f)(x,k) & = \nabla_1(\xi f)(x,k) - ((\nabla_1\xi) f) (x,k) \\
    & = (\alpha\mu_0)\sum_{n\in\integers}f'_n(\mu_0x + \mu_1k - n \hbar' )\xi(x-n\eps', x-nr') \\
    & = (\alpha\mu_0) (\xi \partial_1 f)(x,k),
  \end{align*}
  as well as
  \begin{align*}
  (\xi \delta_2 f)(x,k) & = \nabla_2(\xi f)(x,k) - ((\nabla_2\xi) f) (x,k) \\
    & = (\beta\eps'+\gamma r')\sum_{n\in\integers} n f_n(\mu_0x + \mu_1k - n \hbar' )\xi(x-n\eps', x-nr') \\
                        & = \frac{\beta\eps'+\gamma r'}{2\pi i}
                          (\xi \partial_2 f)(x,k),
  \end{align*}
  which establish the results.
\end{proof}

\noindent
As a corollary we have the condition for $\nabla$ to be a right module connection.

\begin{proposition}\label{prop:right.module.connection}
  Let $\E_{\hbar,\hbar'}$ be a $(\Chinf,\Chpinf)$-bimodule with and
  let $\nabla$ be a left module connection as defined in
  \eqref{eq:bimodule.conn.def}.  If $\alpha=1/\mu_0$ and
  $\beta\eps'+\gamma r'=2\pi i$ then $\nabla$ is a right module
  connection; that is,
  \begin{align*}
    \nabla_k( \xi f) = (\nabla_k\xi) f + \xi(\partial_kf)
  \end{align*}
  for $k=1,2$, $\xi\in\Ehhp$ and $f\in\Chpinf$. 
\end{proposition}

\begin{proof}
  With the restriction on the parameters one has
  $\delta_k(f)=\partial_k(f)$ for $k=1,2$ and $f\in\Chpinf$.  With
  $T=f$ (acting on the right), \eqref{der-conn} becomes
  \begin{align*}
    \xi(\partial_kf) = \nabla_k(\xi f) - (\nabla_k\xi)f,
  \end{align*}
  showing that $\nabla$ is indeed a (right) connection.
\end{proof}

\noindent
In order for $\nabla$ to be a bimodule connection on $\Ehhp$ the parameters, apart
from satisfying the requirements of a bimodule, need to
satisfy the requirements of
Proposition~\ref{prop:left.module.connection} and
Proposition~\ref{prop:right.module.connection}. Let us work out what
this implies for the parameters.

\begin{proposition}\label{prop:bimodule.connection.conditions}
  For $\lambda_0\!\neq\! 0$ and $\hbar,\hbar'\!>\!0$, let $\Ehhp$ be a
  $(\Chinf,\Chpinf)$-bimodule with parameters
  $\lambda_0,\lambda_1,\eps,r$ and $\lambda_0,\mu_1,\eps',r'$
  respectively, satisfying the requirements of
  Proposition~\ref{prop:bimodule}, and assume that
  $\beta,\gamma\in\reals$ are such that
  \begin{align*}
    \beta\eps + \gamma r = \beta\eps'+\gamma r' =2\pi i.
  \end{align*}
  \begin{enumerate}
  \item If $\hbar=\hbar'$ then $\beta=0$, $r=r'\neq 0$, $\eps\neq\eps'$ and
    \begin{align*}
      &\lambda_0 = \frac{\hbar}{\eps'-\eps}\quad
      \lambda_1 = -\frac{\hbar\eps'}{r(\eps'-\eps)}\quad
        \mu_1 = -\frac{\hbar\eps}{r(\eps'-\eps)}\quad
      \gamma=\frac{2\pi i}{r}.
    \end{align*}
  \item If $\hbar\neq \hbar'$ then $\hbar/\hbar'=r/r'\in\rationals$ and
    \begin{align*}
      &\mu_1 = \frac{\hbar'-\hbar}{r'-r}+\lambda_1\quad
      \eps = -\frac{1}{\lambda_0}(\hbar+\lambda_1r)\quad
      \eps' = -\frac{\lambda_1r'}{\lambda_0}\\
      &\beta = 2\pi i\lambda_0\frac{\hbar-\hbar'}{\hbar\hbar'}\quad
        \gamma = \frac{2\pi i}{r'}+2\pi i\lambda_1\frac{\hbar-\hbar'}{\hbar\hbar'}
    \end{align*}
  \end{enumerate}
\end{proposition}

\begin{proof}
  The requirements on the parameters can be summarized as:
  \begin{align}
    &\lambda_0\eps + \lambda_1r = -\hbar\label{eq:lelrmh}\\
    &\lambda_0\eps'+\lambda_1r' = 0\label{eq:leplrp}\\
    &\lambda_0\eps'+\mu_1r'=\hbar'\label{eq:lepmrphp}\\
    &\lambda_0\eps+\mu_1r = 0\label{eq:lemr}\\
    &\beta\eps+\gamma r = \beta\eps'+\gamma r'=2\pi i,\label{eq:begr}
  \end{align}
  and we note that \eqref{eq:lelrmh} and \eqref{eq:lemr} imply
  $(\lambda_1-\mu_1)r=-\hbar$, and \eqref{eq:leplrp} and
  \eqref{eq:lepmrphp} imply that $(\lambda_1-\mu_1)r'=-\hbar'$, giving
  \begin{align}
    (\lambda_1-\mu_1)(r-r') = \hbar'-\hbar.\label{eq:l.minus.m.r.minus.r'}
  \end{align}
  Let us first consider the case when $\hbar=\hbar'$, which implies by
  \eqref{eq:l.minus.m.r.minus.r'} that $\lambda_1=\mu_1$ or
  $r=r'$. However, if $\lambda_1=\mu_1$ then \eqref{eq:lelrmh} and
  \eqref{eq:lemr} imply that $\hbar=0$, which contradicts the
  assumption that $\hbar>0$. Hence, one must have that $r=r'$. Next,
  we note that $r=r'\neq 0$ since otherwise \eqref{eq:lelrmh} and
  \eqref{eq:lemr} imply that $\hbar=0$. The same kind argument shows
  that $\eps'\neq\eps$ if $r=r'$ (otherwise \eqref{eq:lelrmh} and
  \eqref{eq:leplrp} imply that $\hbar=0$). Thus, the following
  equations remain to be solved:
  \begin{align*}
    &\lambda_0\eps+\lambda_1r = -\hbar\\
    &\lambda_0\eps'+\lambda_1r = 0\\
    &\lambda_0\eps'+\mu_1r = \hbar\\
    &\lambda_0\eps+\mu_1r = 0\\
    &\beta\eps + \gamma r = \beta\eps'+\gamma r=2\pi i.
  \end{align*}
  The last equation implies that $\beta(\eps'-\eps) = 0$ which gives
  $\beta=0$ since $\eps\neq\eps'$. Solving the above system for the
  variables $\lambda_0,\lambda_1,\mu_1,\gamma$ immediately gives the
  claimed result.

  Next, let us consider the case when $\hbar\neq\hbar'$. Equation
  \eqref{eq:l.minus.m.r.minus.r'} then implies that
  $\lambda_1\neq\mu_1$ and $r\neq r'$, giving
  \begin{align*}
    \mu_1 = \frac{\hbar'-\hbar}{r-r'}+\lambda_1.
  \end{align*}
  Equations \eqref{eq:lelrmh} and \eqref{eq:lemr} imply that
  $\mu_1r-\lambda_1r = \hbar$ which, by inserting the above
  expression, gives
  \begin{align}\label{rational}
    \frac{\hbar'}{\hbar} = \frac{r'}{r}
  \end{align}
  which in particular implies that $\hbar'\hbar\in\rationals$. With
  $\mu_1=(\hbar'-\hbar)/(r-r')+\lambda_1$ and $\hbar'/\hbar=r'/r$,
  \eqref{eq:lelrmh}--\eqref{eq:begr} are equivalent to
  \begin{align*}
    &\lambda_0\eps+\lambda_1r = -\hbar\\
    &\lambda_0\eps'+\lambda_1r' = 0\\
    &\beta\eps + \gamma r = \beta\eps'+\gamma r'=2\pi i.
  \end{align*}
  Solving the first to equations for $\eps,\eps'$ gives
  \begin{align*}
    &\eps = -\frac{1}{\lambda_0}\paraa{\hbar+\lambda_1r}\mathand
    \eps' = -\frac{\lambda_1r'}{\lambda_0}
  \end{align*}
  which, furthermore, gives
  \begin{align*}
    \beta(\eps-\eps') = \gamma(r'-r)\implies
    \gamma = \frac{\beta}{\lambda_0}\parac{\lambda_1-\frac{\hbar}{r'-r}}.
  \end{align*}
  Inserting the expression for $\gamma$ into
  $\beta\eps+\gamma r=2\pi i$ gives
  \begin{align*}
    \beta = 2\pi i\frac{\lambda_0(r-r')}{\hbar r'}
    =2\pi i\lambda_0\frac{r/r'-1}{\hbar}
    =2\pi i\lambda_0\frac{\hbar-\hbar'}{\hbar\hbar'}
  \end{align*}
  (using $\hbar'/\hbar=r'/r$) and finally
  \begin{align*}
    \gamma = \frac{2\pi i}{r'}+2\pi i\lambda_1\frac{\hbar-\hbar'}{\hbar\hbar'},
  \end{align*}
  completing the proof.
\end{proof}

\noindent
Thus, if $\hbar=\hbar'$ then $\beta=0$ which implies that a
$(\Chinf,\Chinf)$-bimodule connection of the type
\eqref{eq:bimodule.conn.def} has zero curvature
\begin{align*}
  (\nabla_1\nabla_2 -\nabla_2\nabla_1) \xi=\alpha\beta\, \xi = 0.
\end{align*}
In the situation when $\hbar\neq\hbar'$,
Proposition~\ref{prop:bimodule.connection.conditions} shows that in
order for such a bimodule connection to exist, the ratio between $\hbar$
and $\hbar'$ has to be rational (note that, they need not themselves
be rational). Moreover, in this case one may always find module
parameters guaranteeing that $\nabla$ is a bimodule connection. The
curvature then becomes
\begin{align}
  (\nabla_1\nabla_2 - \nabla_2\nabla_1) \xi = \alpha\beta \, \xi
  = 2\pi i \, \frac{\hbar-\hbar'}{\hbar\hbar'}\, \xi.
\end{align}
It is noteworthy that the curvature is independent of the particular
parameters defining the bimodule structure, since it only depends on
$\hbar$ and $\hbar'$. Moreover, since $\hbar/\hbar'$ is rational from \eqref{rational}, the
curvature is a rational multiple of $2\pi i$.

\section{A pseudo-Riemannian calculus}

\noindent
In \cite{aw:cgb.sphere,aw:curvature.three.sphere} the concept of
\emph{pseudo-Riemannian calculi} was introduced in order to discuss
Levi-Civita connections on vector bundles over noncommutative
manifolds. In this particular setting, one can prove that there exists
at most one metric and torsion-free connection on a vector bundle
which is equipped with a soldering map; that is a linear map which maps
derivations into sections of the bundle. In this section, we shall
construct a pseudo-Riemannian calculus for the noncommutative cylinder
and explicitly compute the Levi-Civita connection as well as the
corresponding curvature. 

We recall a few definitions. For the time being, we assume that
$\A$ is an arbitrary $\ast$-algebra.

\begin{definition}\label{def:real.metric.calculus}
  Let $M$ be a (right) $\A$-module and let $h$ be a non-degenerate
  $\A$-valued hermitian form on $M$. Furthermore, let
  $\g\subseteq\Der(\A)$ be a real Lie algebra of hermitian derivations
  of $\A$ and let $\varphi:\g\to M$ be a $\reals$-linear map. The data
  $(M,h,\g,\varphi)$ is called a \emph{real metric calculus over $\A$}
  if
  \begin{enumerate}
  \item the image $\Mphi=\varphi(\g)$ generates $M$ as an $\A$-module,
  \item
    $h(E,E')^\ast=h(E,E')$
    for all $E,E'\in\Mphi$.
  \end{enumerate}
\end{definition}

\noindent
An affine connection $\nabla$ on $(M,\g)$ is a map
$\nabla:\g\times M\to M$ such that (with the notation
$\nabla(\partial, U) = \nabla_{\partial}(U)$), one has
\begin{enumerate}
\item $\nabla_{\partial}(U+V) = \nabla_{\partial}U+\nabla_{\partial}V$,
\item $\nabla_{\lambda \partial+\partial'}U = \lambda\nabla_{\partial}U+\nabla_{\partial'}U$,
\item $\nabla_{\partial}(Ua) = \paraa{\nabla_{\partial}U}a+U\partial(a)$,
\end{enumerate}
for all $U,V\in M$, $\partial,\partial'\in\g$, $a\in\A$ and $\lambda\in\reals$.

\begin{definition}
  Let $(M,h,\g,\varphi)$ be a real metric calculus and let $\nabla$ denote
  an affine connection on $(M,\g)$. If
  \begin{align*}
    h(\nabla_{\partial}E,E') = h(\nabla_{\partial}E,E')^\ast
  \end{align*}
  for all $E,E'\in\Mphi$ and $\partial\in \g$ then $(M,h,\g,\varphi,\nabla)$ is
  called a \emph{real connection calculus}.
\end{definition}

\begin{definition}
  Let $(M,h,\g,\varphi,\nabla)$ be a real connection calculus over $M$. The
  calculus is \emph{metric} if
  \begin{align*}
    \partial\paraa{h(U,V)} = h\paraa{\nabla_{\partial}U,V} + h\paraa{U,\nabla_{\partial} V}
  \end{align*}
  for all $\partial\in\g$, $U,V\in M$, and \emph{torsion-free} if 
  \begin{align*}
    \nabla_{\partial}\varphi(\partial')-\nabla_{\partial'}\varphi(\partial)
    -\varphi\paraa{[\partial,\partial']} = 0
  \end{align*}
  for all $\partial,\partial'\in \g$. A metric and torsion-free real connection
  calculus over $\A$ is called a \emph{pseudo-Riemannian calculus over $\A$}. 
\end{definition}

\noindent
Within this framework, the uniqueness of a metric and torsion-free
connection can be stated in the following way.

\begin{theorem}[\cite{aw:curvature.three.sphere}]
  Let $(M,h,\g,\varphi)$ be a real metric calculus over $\A$. Then there exists
  at most one affine connection $\nabla$ on $(M,\g)$, such that
  $(M,h,\g,\varphi,\nabla)$ is a pseudo-Riemannian calculus. 
\end{theorem}

Let us now return to the noncommutative cylinder. The algebra $\Chinf$
consists of smooth functions on $\reals\times S^1$ that fall off
rapidly at infinity. Of course, there are many more smooth functions
on $\reals\times S^1$ and in this section we allow ourselves to
consider a different algebra $\Chhinf$ consisting of elements
\begin{align*}
  f(u,t) = \sum_{n\in\integers}f_n(u)e^{2\pi int}
\end{align*}
where $f_n\in C^\infty(\reals)$ is such that $f_n\neq 0$
only for a finite number of terms. The product is defined as before as
\begin{align*}
    (f\prodh g)(u,t) = \sum_{n\in\integers}\bracketc{
    \sum_{k\in\integers}f_k(u)g_{n-k}(u+k\hbar)
    }e^{2\pi i nt}
\end{align*}
and the $\ast$-algebra structure is again given as in
Proposition~\ref{prop:star.structure}. Note that whenever
$f,g\in\Chinf\cap\Chhinf$ the products of the two algebras coincide
(as well as the involution). In what follows, we shall construct a pseudo-Riemannian
calculus over $\Chhinf$.

Let $\g$ denote the real (abelian) Lie algebra generated by the hermitian
derivations $\partial_1,\partial_2$, and let $\TChhinf$ denote the free (right)
$\Chhinf$-module of rank 2, with basis $e_1,e_2$. Moreover, we
introduce the hermitian form $h:\TChhinf\times\TChhinf\to\Chhinf$
\begin{align*}
   h(X,Y) = (X^i)^\ast h_{ij}Y^j
\end{align*}
for $X=X^ie_i$ and $Y=Y^ie_i$ and $h_{ij}\in\Chhinf$ such that
$h_{ij}^\ast=h_{ji}=h_{ij}$. (Here and in the following we sum over up-down repeated indices from 1 to 2.)
We shall assume that
$h$ is invertible in the sense that there exists $h^{ij}\in\Chhinf$
such that $h^{ij}h_{jk}=h_{kj}h^{ji}=\delta^i_k$.  (For instance, one
might choose $h_{ij}=e^{2k(u)}\delta_{ij}$ for arbitrary (real)
$k(u)\in\Chhinf$.) Clearly, $h$ is nondegenerate in the sense that
$h(X,Y)=0$ for all $Y\in\TChhinf$ implies that $X=0$. 

\noindent
Finally, we
define $\varphi:\g\to\TChhinf$ as $\varphi(\partial_k) = e_k$ (for $k=1,2$)
extended linearly to all of $\g$. It is immediate that the image of
$\varphi$ generate $\TChhinf$ and that $h(E,E')$ is hermitian for all
$E,E'$ in the image of $\varphi$.  These considerations imply that
$(\TChhinf,h,\g,\varphi)$ is a real metric calculus over
$\Chhinf$. The (unique) Levi-Civita connection can be computed via
Koszul's formula (cf. \cite{aw:cgb.sphere,aw:curvature.three.sphere})
which, since $[\partial_i,\partial_j]=0$, becomes
\begin{align*}
  h(\nabla_ie_j,e_k)
  =\frac{1}{2}\paraa{\partial_ih(e_j,e_k)+\partial_jh(e_k,e_i)-\partial_kh(e_i,e_j)},
\end{align*}
where $\nabla_i=\nabla_{\partial_i}$. Writing
$\nabla_{i}e_j=e_l\Gamma^l_{ij}$ gives
\begin{align*}
  (\Gamma^l_{ij})^\ast h_{lk}
  =\frac{1}{2}\paraa{\partial_ih_{jk}+\partial_jh_{ki}-\partial_kh_{ij}},
\end{align*}
and one finds that
\begin{align*}
  \Gamma^l_{ij} = \frac{1}{2}h^{lk}
  \paraa{\partial_ih_{jk}+\partial_jh_{ki}-\partial_kh_{ij}}.
\end{align*}
Let us exemplify this construction for
$h_{ij}=e^{2k(u)}\delta_{ij}$. In this case, one obtains
\begin{align*}
  &\Gamma^1_{11} = \Gamma^2_{12}=\Gamma^2_{21}=-\Gamma^1_{22}=k'(u)\\
  &\Gamma^2_{11} = \Gamma^1_{12}=\Gamma^1_{21}=\Gamma^2_{22} = 0
\end{align*}
giving
\begin{align*}
  \nabla_1e_1 = e_1k'(u)\qquad
  \nabla_1e_2 = \nabla_2e_1 = e_2k'(u)\qquad
  \nabla_2e_2 = -e_1k'(u).
\end{align*}
We immediately note that
\begin{align*}
  h(\nabla_ie_j,\nabla_ke_l)^\ast =
  h(\nabla_ie_j,\nabla_ke_l) 
\end{align*}
since both $\nabla_ie_j$ and $h_{ij}$ depend only on $u$. This implies
that the curvature of the connection will have all the classical
symmetries \cite{aw:curvature.three.sphere}; hence, there is only one
independent component of the curvature, and one finds
\begin{align*}
  &R(\partial_1,\partial_2)e_1 = \nabla_1\nabla_2e_1-\nabla_2\nabla_1e_1 = e_2k''(u)\\
  &R(\partial_1,\partial_2)e_2 = \nabla_1\nabla_2e_2-\nabla_2\nabla_1e_2 = -e_1k''(u)\\
  &R_{1212} = h\paraa{e_1,R(\partial_1,\partial_2)e_2} = -e^{2k(u)}k''(u),
\end{align*}
giving the Gaussian curvature as
\begin{align*}
  K = \frac{1}{2}h^{ij}R_{ikjl}h^{kl} = -e^{-2k(u)}k''(u).
\end{align*}
For a metric of the above form, a natural integration measure
corresponding to the volume form is given by
$\tau_h(f) = \tau(fe^{2k(u)})$. For the sake of illustration, let us
compute the total curvature (when it exists)
\begin{align*}
  \tau_h(K) &= -\int_{-\infty}^\infty e^{-2k(u)}k''(u)e^{2k(u)}du
  =-\int_{-\infty}^\infty k''(u)du\\
  &= \lim_{u\to-\infty}k'(u)-\lim_{u\to\infty}k'(u)
\end{align*}
Here one notes a certain independence of the total curvature
with respect to perturbations of the metric; i.e. for
$\tilde{k}(u)=\delta(u)+k(u)$ one finds that
$\tau_h(\tilde{K})=\tau_h(K)$ whenever
\begin{align*}
 \lim_{u\to\infty}\delta'(u)=\lim_{u\to-\infty}\delta'(u). 
\end{align*}
For instance, for $k(u)=\ln(\cosh(u))$, corresponding to the
induced metric on the catenoid (cf. \cite{ah:catenoid}), one obtains
\begin{align*}
  \tau_h(K) = \lim_{u\to-\infty}\tanh(u) - \lim_{u\to\infty}\tanh(u) = -2
\end{align*}
which, in the geometrical situation where the trace naturally gains a factor of
$2\pi$ (from the integration along $S^1$), gives the expected value
of $-4\pi$ for the total curvature of the catenoid.

\subsection*{Acknowledgments}

\bigskip
\noindent
JA is supported by the Swedish Research Council.  GL is partially
supported by INFN, Iniziativa Specifica ``Gauge and String Theories
(GAST)", and by the ``National Group for Algebraic and Geometric
Structures and their Applications (GNSAGA - INdAM)", and ``Laboratorio
Ypatia di Scienze Matematiche (LIA-LYSM)" CNRS - INdAM.

\end{document}